\documentclass[a4]{SCIYA2017enOL}

\online

\newcommand{\res}{{\rm res}}
\def\pt#1{\bm{#1}}
\newcommand{\ideal}[1]{\mathcal{#1}}

\usepackage{amsfonts,amssymb}
\usepackage{mathrsfs}
\usepackage{amsmath,amssymb,bbding,ebezier}
\usepackage{amsthm}
\usepackage{graphicx,float,tabularx}
\usepackage{enumerate}
\usepackage{multicol}
\usepackage{color}
\usepackage{extarrows}
\usepackage{bm}

\begin{document}

%\ensubject{fdsfd}%
\ensubject{Computer Mathematics}

%%%%%%%%%%%%%%%%%%%%%%%%%%%%%%%%%%%%%%%%%%%%%%%%%%%%%%%
%%% Authors do not modify the information below

\ArticleType{}%{ARTICLES}%栏目
%\SpecialTopic{Progress of Projects Supported by NSFC}%专题
%\SubTitle{Dedicated to Professor Yang Lo on the Occasion of his {\rm 70}th Birthday}%专刊说明
\Year{2017}
\Month{January}%
\Vol{60}
\No{1}
\BeginPage{1} %
\DOI{}
\ReceiveDate{}
\AcceptDate{}
%\OnlineDate{January 1, 2017}
%%%%%%%%%%%%%%%%%%%%%%%%%%%%%%%%%%%%%%%%%%%%%%%%%%%%%%%

%%% title:
%%%   \title{title}{title for citation}
\title{The second discriminant of a univariate polynomial\footnotemark[4]\footnotetext[4]{This paper has been accepted for publication in SCIENCE CHINA Mathematics.}}{}

%%% Corresponding author
%%%%   \author[number]{Full name}{{email@xxx.com}}
%%% General author
%%%   \author[number]{Full name}{}

\author[1,2,3]{Dongming Wang}{{Dongming.Wang@lip6.fr}}
\author[2,$\ast$]{Jing Yang}{{yangjing0930@gmail.com}}

%%% Author information for page head.
%\AuthorMark{Author A}

%%% Authors for citation.
%\AuthorCitation{Author A, Author B, Author C}

%%% Address.
%%%   \address[number]{Address, City {\rm Postcode}, Country}
\address[1]{BDBC--LMIB--School of Mathematics and Systems Science, Beihang University, Beijing {\rm 100191}, China}
\address[2]{SMS--HCIC, Guangxi University for Nationalities, Nanning {\rm 530006}, China}
\address[3]{Centre National de la Recherche Scientifique, Paris {\rm 75794}, France}

%%% Abstract.
\abstract{We define the second discriminant $D_2$ of a univariate polynomial $f$ of degree greater than $2$ as the product of the linear forms $2\,r_k-r_i-r_j$ for all triples of roots $r_i, r_k, r_j$ of $f$ with $i<j$ and $j\neq k, k\neq i$. $D_2$ vanishes if and only if $f$ has at least one root which is equal to the average of two other roots. We show that $D_2$ can be expressed as the resultant of $f$ and a determinant formed with the derivatives of $f$, establishing a new relation between the roots and the coefficients of $f$. We prove several notable properties and present an application of $D_2$.}

%  Keyword.

%%% MSC numbers. Requested items.

\maketitle

%%%%%%%%%%%%%%%%%%%%%%%%%%%%%%%%%%%%%%%%%%%%%%%%%%%%%%
%% The main text.
%%%%%%%%%%%%%%%%%%%%%%%%%%%%%%%%%%%%%%%%%%%%%%%%%%%%%%

%====================================================================================
\section{Introduction}
\label{sec:introduction}
%====================================================================================
The discriminant of a univariate polynomial $f=f(x)$ may be defined as a function of the coefficients of $f$ in $x$, whose vanishing is a necessary and sufficient condition for $f$ to have multiple roots for $x$. The term {\em discriminant} was used early by Sylvester in \cite{S1851O} and it will be referred  to as the \emph{first discriminant} hereinafter. The first discriminant of $f$ contains information about the nature of the roots\footnote{For example, if the discriminant of a cubic polynomial $f$ with real coefficients is positive, then $f$ has no complex root \cite{G1990S}.} of $f$ and has played a fundamental role in the study of polynomial equations. It has many remarkable properties \cite{C1993R,GKZ1994D} and has been used in diverse areas ranging from algebraic geometry and Galois theory to bifurcation analysis and number theory.

To define the first discriminant $D_1$ of $f$, one considers the simple form $r_i-r_j$ for any pair of roots $r_i, r_j$ of $f$ with $i\neq j$ and takes the product of all such forms as $D_1$, which can be expressed as the resultant of $f$ and its derivative. In this paper, we define the {\em second discriminant} $D_2$ of $f$ (of degree greater than $2$) as the product of the linear forms
$2\,r_k-r_i-r_j$ for all triples of roots $r_i, r_k, r_j$ of $f$ with $i<j$ and $j\neq k$, $k\neq i$.

More concretely, let
\begin{equation}\label{eqF}
f=x^n+a_{n-1}x^{n-1}+\cdots +a_1x+a_0
\end{equation}
be any univariate polynomial
of degree $n\geq 3$ in $x$ with real or complex coefficients. Let $r_1,\ldots,r_n$ be the $n$ roots of $f$ for $x$ over ${\Bbb C}$, the field of complex numbers. By a \emph{symmetric triple} of roots, we mean a triple $(r_i,r_k,r_j)$ of roots of $f$ with $i<j$ and $j\neq k$, $k\neq i$ such that $r_k=(r_i+r_j)/2$.
Then, obviously, $D_2=0$ if and only if $f$ has a symmetric triple of roots.
We will show that $D_2$ can be expressed as the resultant of $f$ and a determinant formed with the derivatives of $f$, and thus as a polynomial in $a_0,\ldots,a_{n-1}$ with rational coefficients.
Several other properties of $D_2$ will also be proved, highlighting the geometric interest of the symmetric triples of roots.
The second discriminant $D_2$ complements the well-known first discriminant $D_1$ of $f$ in depicting the structural properties such as distribution, position, and configuration of the roots of $f$.

In the following section, the second discriminant $D_2$ for an arbitrary univariate polynomial $f$ of degree $n$ is defined formally in terms of the roots of $f$;
some simple properties of $D_2$ are then proved. In Sections \ref{sec:Construction} and \ref{sec:Properties}, we show that $D_2$ as a polynomial in the coefficients of $f$ is irreducible of total degree $3\,(n-1)(n-2)/2$.  In Sections \ref{sec:resultant} and \ref{sec:ideals}, we elaborate $D_2$ with resultants and ideals from the perspective of modern algebra, which leads to different ways for the construction of $D_2$. In Section \ref{sec:determinant}, we provide exact formulas for the degrees of some determinant polynomials involved in the construction of $D_2$. Finally, an application of $D_2$ to the classification of root configurations is presented and the paper is concluded with some remarks in Section \ref{sec:ApplicationRemarks}.

%====================================================================================
\section{Symmetric Triples of Roots and the Second Discriminant}
\label{sec:Definition}
%====================================================================================
Let $f\in \mathbb{C}[x]$ be as in \eqref{eqF} with $\deg(f, x)=n\geq 3$ and $r_1, \ldots, r_n$ be the $n$ roots of $f$ over $\mathbb{C}$ as above. Consider any two roots $r_i$ and $r_j$. We call $(r_i+r_j)/2$ the \emph{average} of $r_i$ and $r_j$. For any triple $\bm{r}=(r_i, r_k, r_j)$, where
\[f(r_i)=f(r_j)=f(r_k)=0 \quad\mbox{and}\quad i< j, j\neq k, k\neq i,\]
if $r_k$ is the average of $r_i$ and $r_j$, i.e., $r_k=(r_i+r_j)/2$, then $\bm{r}$ is called a \emph{symmetric triple} of roots of $f$. We are interested in the condition under which $f$ has symmetric triples of roots.

Recall that the first discriminant of $f$ may be defined as
\[
D_1= \prod_{1\leq i<j\leq n}(r_i-r_j)^2=\pm\prod_{\scriptsize{\begin{array}{c}
1\leq i,j\leq n\\
i\neq j
\end{array}}}(r_i-r_j).
\]
$D_1=0$ if and only if $f$ has a multiple root.
To obtain the condition under which $f$ has a symmetric triple of roots, we define the second discriminant $D_2$ of $f$ as follows:
\begin{equation}\label{eq:D2}
D_2 = \prod_{
\scriptsize{\begin{array}{c}
1\leq i,j,k\leq n\\
i< j,j\neq k,k\neq i
\end{array}}
}(2\,r_k-r_i-r_j),
\end{equation}
a symmetric polynomial of total degree $n(n-1)(n-2)/2$ in $r_1,\ldots,r_n$. For the sake of simplicity, we shall write $i< j\neq k$ for the range of $i,j,k$ determined by $1\le i,j,k\leq n$ and $i< j$, $j\neq k$, $k\neq i$.
\begin{remark}
$D_1=0$ does not imply $D_2=0$, and vice versa.
\end{remark}

\begin{proposition}
\begin{enumerate}[$($a$)$]
\item If $D_2\neq 0$, then any root of $f$ has multiplicity not greater than $2$.
\item If $D_1\neq 0$ and $D_2=0$, then there exist pairwise distinct $r_i, r_j, r_k$ with $i<j\neq k$ such that $2\,r_k-r_i-r_j=0$.
\end{enumerate}
\end{proposition}

\begin{proof}
(a) Suppose that $f$ has a root with multiplicity greater than $2$; then $r_i=r_j=r_k~(i< j\neq k)$ for some $i,j,k$. This is a special case of
$r_k=(r_i+r_j)/2$,
so $D_2=0$, which leads to contradiction.

(b) $D_1\neq 0$ implies that $r_i, r_j, r_k$ are pairwise distinct for any $i<j\neq k$ and $D_2=0$ implies the existence of $r_i, r_j, r_k$ with $i<j\neq k$ such that $2\,r_k-r_i-r_j=0$.
\end{proof}

\begin{theorem}
$D_2=0$ if and only if $f$ has a symmetric triple of roots.
\end{theorem}

\begin{proof}
($\Longrightarrow$) $D_2=0$ implies that there exist $r_i$, $r_j$, $r_k$ such that $r_k=(r_i+r_j)/2$. Thus $(r_i, r_k, r_j)$ is a symmetric triple of roots which we seek for.

($\Longleftarrow$) Suppose that $(r_i, r_k, r_j)$ is a symmetric triple of roots that $f$ has. Then $r_k=(r_i+r_j)/2$. It follows that
\[D_2=\prod_{i< j\neq k}(2\,r_k-r_i-r_j)=0.\]

\vspace{-0.5cm}
\end{proof}

The second discriminant $D_2$ defined above is a polynomial in the roots $r_1,\ldots,r_n$ of $f$. This polynomial is symmetric with respect to the roots, so $D_2$ can be expressed as another polynomial in the coefficients $a_0,\ldots,a_{n-1}$ of $f$. We will provide explicit formulas and simple algorithmic approaches for the construction of the polynomial in $a_i$, together with several properties about $D_2$.

%====================================================================================
\section{Expression of the Second Discriminant}
\label{sec:Construction}
%====================================================================================
In this section, we show that the second discriminant $D_2$ of $f$ can be expressed as a polynomial in $a_0,\ldots,a_{n-1}$, the coefficients of $f$. The expression of $D_2$ we have discovered as the resultant of $f$ and the determinant of a shifting matrix $H$ formed with the derivatives $f^{(1)}, \ldots, f^{(n)}$ of $f$, given in the following theorem, appears pretty amazing. It is puzzling how and why the derivatives of $f$ get occurred in $H$ so structurally. We will answer this question in Lemma~\ref{lem:resH}
by linking $H$ to the resultant of two other polynomials derived from $f$.

As usual, denote by $\det(M)$ the determinant of any square matrix $M$ and
by $\res(f, g, x)$ the Sylvester resultant of any two polynomials $f$ and $g$ with respect to $x$.

\begin{theorem}\label{thm:D2} The second discriminant $D_2$ of $f$ is equal to the
resultant of $f$ and a determinant $H$ formed with the derivatives of $f$ with respect to $x$. More precisely,
\[
D_2=\res(f, H, x),
\]
where $H$ is the $(n-2)$th leading principal minor of the following matrix
\begin{equation}\label{matM}
M=\left(
\begin{array}{cccccccc}
\frac{f^{(2)}}{2!} &  \frac{f^{(4)}}{4!} &  \frac{f^{(6)}}{6!} & \cdots \!\!&\! \cdots & \frac{f^{(2 l )}}{(2 l )!}  & \cdots  \!\!&\! \cdots\\ [6pt]
\frac{f^{(1)}}{1!} &  \frac{f^{(3)}}{3!} &  \frac{f^{(5)}}{5!} & \cdots \!\!&\! \cdots & \frac{f^{(2 l -1)}}{(2 l -1)!}  & \cdots  \!\!&\! \cdots\\ [6pt]
0 & \frac{f^{(2)}}{2!} &  \frac{f^{(4)}}{4!} &  \cdots \!\!&\! \cdots & \frac{f^{(2 l -2)}}{(2 l -2)!} & \cdots  \!\!&\! \cdots\\ [6pt]
0 & \frac{f^{(1)}}{1!} &  \frac{f^{(3)}}{3!} &  \cdots \!\!&\! \cdots & \frac{f^{(2 l -3)}}{(2 l -3)!}  & \cdots  \!\!&\! \cdots\\ [6pt]
0 & 0 & \frac{f^{(2)}}{2!} &\cdots \!\!&\! \cdots &\frac{f^{(2 l -4)}}{(2 l -4)!} & \cdots  \!\!&\! \cdots\\ [0pt]
\vdots & \vdots & \vdots & \ddots \!\!&\! &\vdots & \ddots  \!\!&\! \\[-11pt]
\vdots & \vdots & \vdots & \!\!&\!\hspace{-2pt}\raisebox{0.10cm}{\mbox{$\ddots$}} & \vdots & \!\!&\!\hspace{-2pt}\raisebox{0.10cm}{\mbox{$\ddots$}}
\end{array}
\right)
\end{equation}
and $f^{(\mu )}$ denotes the $\mu $th derivative of $f$.
\end{theorem}

Note that
\[\res(f,H,x)=\prod_{k=1}^nH(r_k),\]
where $r_1,\ldots,r_n$ are the $n$ roots of $f$ as before.
To prove Theorem~\ref{thm:D2}, we only need to show that for each $k$, $H(r_k)$ is the product of $2\,r_k-r_i-r_j$ for all $i,j$ with $i< j\neq k$. The proof will be divided into two parts. In the first part, it is shown that for any $i,j$ with $i< j\neq k$, $2\,r_k-r_i-r_j$ is a divisor of $H(r_k)$ (see Lemmas \ref{lem:rootH} and \ref{lem:division}). The second part is devoted to proving that the leading term of $H(r_k)$ with respect to $r_k$ is $(2\,r_k)^{\frac{(n-1)(n-2)}{2}}$
(see Lemma \ref{lem:initerm}).

\begin{lemma}\label{lem:rootH}
If $r_k=(r_i+r_j)/2$ for $i< j\neq k$, then $H(r_k)=0$.
\end{lemma}

\begin{proof}
It suffices to show that the lemma holds for $k=1$, $i=2$, and $j=3$.
Denote by $\Omega_{l}^{\gamma}$ the set of all $\gamma$-tuples obtained from $(l,\ldots,n)$ by deleting $n-\gamma$ components, where $l$ is a positive integer not greater than $n$.
Let $b_\mu =x-r_\mu $ for $\mu =1,\ldots,n$. By calculus, it is easy to verify that
\begin{equation}\label{eq:Fkr}
\frac{f^{(\nu )}}{\nu \,!}=\left\{
\begin{array}{cl}
\sum_{(\mu _1,\ldots,\mu _{n-\nu })\in\Omega_1^{n-\nu }}{b_{\mu _1}\cdots b_{\mu _{n-\nu }}},&\nu =0,\ldots, n-1;\\[16pt]
1,&\nu =n.
\end{array}
\right.
\end{equation}

Let $c_\mu =r_1-r_\mu $ for $\mu =2,\ldots,n$ and suppose that $r_1=(r_2+r_3)/2$. Then $c_2+c_3=0$. Substituting $x=r_1$ into \eqref{eq:Fkr} and observing that any term ${b_{\mu _1}\cdots b_{\mu _{n-\nu }}}$ involving $x-r_1$ vanishes at $x=r_1$, we have
\begin{align*}
\frac{f^{(\nu )}}{\nu \,!}\Big|_{x=r_1}&=\sum\limits_{(\mu _1,\ldots,\mu _{n-\nu })\in\Omega_1^{n-\nu }}{b_{\mu _1}\cdots b_{\mu _{n-\nu }}}\Big|_{x=r_1}\\[10pt]
&=t_{n-\nu }+c_2c_3t_{n-\nu -2}+c_2t_{n-\nu -1}+c_3t_{n-\nu -1}\\[10pt]
&=c_2c_3t_{n-\nu -2}+(c_2+c_3)t_{n-\nu -1}+t_{n-\nu }\\[10pt]
&=-c_2^2t_{n-\nu -2}+t_{n-\nu },
\end{align*}
where
\[t_{n-\nu }=\left\{
\begin{array}{cl}
\sum\limits_{(\mu _1,\ldots,\mu _{n-\nu })\in\Omega_4^{n-\nu }}{c_{\mu _1}\cdots c_{\mu _{n-\nu }}}&\mbox{if}\,~3\leq \nu \leq n-1;\\[10pt]
1&\mbox{if}\,~\nu =n;\\[10pt]
0&\mbox{if}\,~\nu \leq 2\mbox{~or~} \nu \geq n+1.
\end{array}
\right.
\]

Substitution of $t_{n-\nu }$ into $H(r_1)$ yields
\[\begin{array}{l}
H(r_1)=\left|
\begin{array}{cccccc}\smallskip
-c_2^2t_{n-4} & -c_2^2t_{n-6}+t_{n-4} &  -c_2^2t_{n-8}+t_{n-6} &  \cdots \!\!&\! \cdots & -c_2^2t_{n-2\nu -2}+t_{n-2\nu }  \\ [5pt]
-c_2^2t_{n-3} &  -c_2^2t_{n-5}+t_{n-3} & -c_2^2t_{n-7}+t_{n-5} &  \cdots \!\!&\! \cdots & -c_2^2t_{n-2\nu -1}+t_{n-2\nu +1} \\ [5pt]
0 & -c_2^2t_{n-4} &  -c_2^2t_{n-6}+t_{n-4} &  \cdots \!\!&\! \cdots  &  -c_2^2t_{n-2\nu }+t_{n-2\nu +2} \\ [5pt]
0 & -c_2^2t_{n-3} &  -c_2^2t_{n-5}+t_{n-3} &  \cdots \!\!&\! \cdots  &  -c_2^2t_{n-2\nu +1}+t_{n-2\nu +3} \\ [5pt]
0 & 0 & -c_2^2t_{n-4}+t_{n-2} & \cdots \!\!&\!  \cdots & -c_2^2t_{n-2\nu +2}+t_{n-2\nu +4} \\ [5pt]
\vdots &\vdots &  \vdots & \ddots \!\!\!&\!\!\! &\vdots \\[-11pt]
\vdots &\vdots &  \vdots &  \!\!\!&\!\!\! \hspace{3pt}\raisebox{0.12cm}{\mbox{$\ddots$}} &\vdots \\[5pt]
0 & 0 & 0 & \cdots  \!\!&\! \cdots  & 0 \\ [5pt]
0 & 0 & 0 & \cdots  \!\!&\! \cdots  & 0 \\ [5pt]
0 & 0 & 0 & \cdots  \!\!&\! \cdots  & 0 \\ [5pt]
0 & 0 & 0 & \cdots  \!\!&\! \cdots  & 0
\end{array}
\right.
\end{array}
\]

\[
\qquad\qquad\qquad\quad\qquad\qquad
\begin{array}{l}
\left.
\begin{array}{cccccccc}
\smallskip
\cdots \!&\! \cdots & 0 & \cdots \!&\! \cdots & 0 & 0 & 0\\ [5pt]
\cdots \!&\! \cdots & 0 & \cdots \!&\! \cdots & 0 & 0 & 0\\ [5pt]
\cdots \!&\! \cdots & 0 & \cdots \!&\! \cdots & 0 & 0 & 0\\ [5pt]
\cdots \!&\! \cdots & 0 & \cdots \!&\! \cdots & 0 & 0 & 0 \\ [5pt]
\cdots \!&\! \cdots & 0 & \cdots \!&\! \cdots & 0 & 0 & 0 \\ [5pt]
\ddots \!&\! & \vdots &\ddots \!&\! & \vdots & \vdots & \vdots \\[-10pt]
\smallskip
\!&\! \hspace{-4pt}\raisebox{0.1cm}{\mbox{$\ddots$}} & \vdots & \!&\! \hspace{-4pt}\raisebox{0.1cm}{\mbox{$\ddots$}} & \vdots & \vdots & \vdots \\[5pt]
\cdots \!&\! \cdots & -c_2^2t_{2\nu -6}+t_{2\nu -4} &\cdots \!&\! \cdots & -c_2^2+t_2 & 1  & 0\\ [5pt]
\cdots \!&\! \cdots & -c_2^2t_{2\nu -5}+t_{2\nu -3} & \cdots \!&\!\cdots & -c_2^2t_1+t_3 & t_1  & 0\\ [5pt]
\cdots \!&\! \cdots & -c_2^2t_{2\nu -4}+t_{2\nu -2} & \cdots \!&\!\cdots & -c_2^2t_2+t_4    & -c_2^2+t_2  & 1\\ [5pt]
\cdots \!&\! \cdots & -c_2^2t_{2\nu -3}+t_{2\nu -1} & \cdots \!&\!\cdots & -c_2^2t_3+t_5 & -c_2^2t_1+t_3  & t_1
\end{array}
\right|.
\end{array}
\]

For each $\mu =n-2,\ldots,2$, add the $\mu $th column multiplied by $c_2^2$ to the $(\mu -1)$th column of $H(r_1)$ iteratively. It follows that

\[
H(r_1)=\left|
\begin{array}{ccccccccccc}
0 & t_{n-4} & t_{n-6} &  \cdots \!&\!  \cdots & t_{n-2\nu }  & \cdots \!&\! \cdots & 0 & 0\\ [5pt]
0 & t_{n-3} & t_{n-5} &  \cdots \!&\!  \cdots & t_{n-2\nu +1}  & \cdots \!&\! \cdots& 0 & 0\\ [5pt]
0 & 0 &  t_{n-4} &  \cdots \!&\! \cdots & t_{n-2\nu +2} & \cdots \!&\! \cdots& 0 & 0\\ [5pt]
0 & 0 &  t_{n-3} &  \cdots \!\!\!&\!\!\! \cdots & t_{n-2\nu +3} & \cdots \!\!\!&\!\!\! \cdots& 0 & 0\\ [5pt]
\vdots &\vdots & \vdots & \ddots \!&\! & \vdots &\ddots \!&\!& \vdots & \vdots\\ [-11pt]
\vdots &\vdots & \vdots & \!&\!\hspace{-5pt}\raisebox{0.1cm}{\mbox{$\ddots$}} & \vdots &\!&\!\hspace{-4pt}\raisebox{0.1cm}{\mbox{$\ddots$}} & \vdots & \vdots\\
0 & 0 & 0  & \cdots \!&\! \cdots & \vdots & \cdots \!&\!\cdots & 1  & 0\\ [5pt]
0 & 0 & 0  & \cdots \!&\! \cdots & \vdots & \cdots \!&\!\cdots & t_1  & 0\\ [5pt]
0 & 0 & 0  & \cdots \!&\! \cdots & \vdots & \cdots \!&\!\cdots & t_2  & 1\\ [5pt]
0 & 0 & 0  & \cdots \!&\! \cdots & \vdots & \cdots \!&\!\cdots & t_3  & t_1
\end{array}
\right|=0.
\]

\vspace{-0.5cm}
\end{proof}

\begin{lemma}\label{lem:division}
For any $i< j\neq k$, the linear form $2\,r_k-r_i-r_j$ divides $H(r_k)$.
\end{lemma}

\begin{proof}
Let $r_k=(r_i+r_j)/2$, where $i$ and $j$ are arbitrary but fixed. By Lemma~\ref{lem:rootH},
$$H(r_k)=0.$$
It follows that
\[[r_k-(r_i+r_j)/2]\mid H(r_k), \mbox{~~~or~~~} (2\,r_k-r_i-r_j)\mid H(r_k)\]
over ${\Bbb Q}$ (the field of rational numbers).
\end{proof}

Note that $u\mid v$ stands for ``$u$ divides $v$'' as usual.
Let $c_\mu =r_1-r_\mu $ for $\mu =2,\ldots,n$. It is easy to verify that
\begin{equation}
\frac{f^{(\nu )}}{\nu \,!}\Bigg|_{x=r_1}=t^*_{n-\nu },
\end{equation}
where
\[
t^*_{n-\nu }=\left\{
\begin{array}{cl}
\sum\limits_{(\mu _1,\ldots,\mu _{n-\nu })\in\Omega_2^{n-\nu }}{c_{\mu _1}\cdots c_{\mu _{n-\nu }}}&\mbox{if}\,~1\leq \nu \leq n-1;\\[12pt]
1&\mbox{if}\,~\nu =n.
\end{array}
\right.
\]
Let $M(r_1)$ be the matrix obtained from $M$ in \eqref{eq:D2} by replacing $x$ with $r_1$. Then
\begin{equation*}
M(r_1)=\left(
\begin{array}{cccccccc}
t^*_{n-2} &  t^*_{n-4} &  t^*_{n-6} & \cdots \!&\! \cdots & t^*_{n-2\nu }  & \cdots \!&\! \cdots\\ [5pt]
t^*_{n-1} &  t^*_{n-3} &  t^*_{n-5} & \cdots \!&\! \cdots & t^*_{n-2\nu +1}  & \cdots \!&\! \cdots\\ [5pt]
0 & t^*_{n-2} &  t^*_{n-4}  & \cdots \!&\! \cdots & t^*_{n-2\nu +2} & \cdots \!&\! \cdots \\ [5pt]
0 & t^*_{n-1} &  t^*_{n-3}  & \cdots \!&\! \cdots & t^*_{n-2\nu +3}  & \cdots \!&\! \cdots\\ [5pt]
\vdots & \vdots & \vdots &\ddots \!&\!  &\vdots & \ddots \!&\!\\[-11pt]
\vdots & \vdots & \vdots & \!&\! \hspace{-5pt}\raisebox{0.1cm}{\mbox{$\ddots$}} &\vdots &\!&\!\hspace{-5pt}\raisebox{0.1cm}{\mbox{$\ddots$}}
\end{array}
\right).
\end{equation*}
Since $C_{n-1}^{n-\nu }r_1^{n-\nu }$ is the leading coefficient of $t^*_{n-\nu }$ with respect to $r_1$,
$t^*_{n-\nu }$ can be written as
\[
t^*_{n-\nu }=C_{n-1}^{n-\nu }r_1^{n-\nu }+\mathcal{O}(r_1^{n-\nu }),\]
where $\mathcal{O}(r_1^{n-\nu })$ denotes terms of degree less than $n-\nu $ in $r_1$.
Now let $M_n(r_1)$ be the $(n-2)$th leading principal minor of the matrix obtained from $M(r_1)$ by replacing each entry $t^*_{n-\nu }$ with $C_{n-1}^{n-\nu }r_1^{n-\nu }$. Then
\[
M_n(r_1)=\Big|
\pt{m}_1(r_1),\ldots,\pt{m}_\mu (r_1),\ldots,\pt{m}_{n-2}(r_1)
\Big|,
\]
where
\[
\pt{m}_\mu (r_1)=\left\{
\begin{array}{l}
\Big(C_{n-1}^{n-2\mu }r^{n-2\mu }_1,\ldots,C_{n-1}^{n-2\mu +\nu }r^{n-2\mu +\nu }_1,\ldots, \underbrace{0,\ldots\ldots\ldots\ldots,0}_{\max(n-2\mu -2,0)\mbox{~terms}}\Big)^T\\[25pt]
\qquad\qquad\qquad\qquad\mbox{if}\,~\mu \leq \left\lceil \dfrac{n-2}{2}\right\rceil\mbox{~and~}0\leq \nu \leq \min(n-3,2\,\mu -1);\\[25pt]
\Big(\underbrace{0,\ldots\ldots\ldots,0}_{2\mu -n\mbox{~terms}},C_{n-1}^{0}r_1^0,\ldots,C_{n-1}^{\nu }r_1^\nu ,\ldots\Big)^T\qquad\\[25pt]
\qquad\qquad\qquad\qquad\mbox{if}\,~\mu > \left\lceil \dfrac{n-2}{2}\right\rceil\mbox{~and~}0\leq \nu \leq 2\,n-2\,\mu -3.
\end{array}
\right.
\]
Therefore, $M_n(r_1)$ has the following form:
\[\begin{array}{l}\medskip
M_n(r_1)=\left|
\begin{array}{cccccccc}\smallskip
C_{n-1}^{n-2}r_1^{n-2} &  C_{n-1}^{n-4}r_1^{n-4} &  C_{n-1}^{n-6}r_1^{n-6} & \cdots \!&\! \cdots &  C_{n-1}^{n-2\nu }r_1^{n-2\nu } & \cdots \!&\! \cdots \\ [8pt]
C_{n-1}^{n-1}r_1^{n-1} &  C_{n-1}^{n-3}r_1^{n-3} &  C_{n-1}^{n-5}r_1^{n-5} & \cdots \!&\! \cdots &  C_{n-1}^{n-2\nu +1}r_1^{n-2\nu +1} & \cdots \!&\! \cdots \\ [8pt]
0 & C_{n-1}^{n-2}r_1^{n-2} &  C_{n-1}^{n-4}r_1^{n-4} &  \cdots \!&\!  \cdots &  C_{n-1}^{n-2\nu +2}r_1^{n-2\nu +2} & \cdots \!&\! \cdots \\ [8pt]
0 & C_{n-1}^{n-1}r_1^{n-1} &  C_{n-1}^{n-3}r_1^{n-3} &  \cdots \!&\!  \cdots &  C_{n-1}^{n-2\nu +3}r_1^{n-2\nu +3} & \cdots \!&\! \cdots \\ [8pt]
0 & 0 & C_{n-1}^{n-2}r_1^{n-2} &\cdots \!&\!  \cdots &  C_{n-1}^{n-2\nu +4}r_1^{n-2\nu +4}   \\ [8pt]
\vdots & \vdots & \vdots &\ddots  \!&\!   &  \vdots &\ddots  \!&\!     \\ [-11pt]
\vdots & \vdots & \vdots &  \!&\!  \hspace{-6pt}\raisebox{0.1cm}{\mbox{$\ddots$}} &  \vdots &  \!&\!  \hspace{-4pt}\raisebox{0.1cm}{\mbox{$\ddots$}}
\end{array}
\right|.
\end{array}
\]
Apparently, the above expression for $M_n(r_1)$ remains valid when $r_1$ is substituted by $r_k$ for any $k>1$.

\begin{lemma}\label{lem:initerm}
$M_n(r_k)=(2\,r_k)^{\frac{(n-1)(n-2)}{2}}$ for $k=1,\ldots,n$.
\end{lemma}

\begin{proof}
We prove the lemma for $k=1$. The proof applies for any $k\neq 1$. Substitution of $C_n^\mu =C_{n-1}^{\mu }+C_{n-1}^{\mu -1}$ into $\pt{m}_\mu (r_k)$ yields
\[
\pt{m}_\mu (r_k)=\left\{
\begin{array}{l}
\Big((C_{n-2}^{n-2\mu }+C_{n-2}^{n-2\mu -1})r_k^{n-2\mu },\ldots,(C_{n-2}^{n-2\mu +\nu }+C_{n-2}^{n-2\mu +\nu -1})r_k^{n-2\mu +\nu },\\ [12pt]
\qquad\qquad\qquad\qquad\qquad\qquad\qquad\qquad\qquad\quad \ldots, \underbrace{0,\ldots\ldots\ldots\ldots,0}_{\max(n-2\mu -2,0)\mbox{~terms}}\Big)^T\\[25pt]
\qquad\qquad\qquad\mbox{if}\,~\mu \leq \left\lceil \dfrac{n-2}{2}\right\rceil\mbox{~and~}0\leq \nu \leq \min(n-3,2\,\mu -1);\\[10pt]
\Big(\underbrace{0,\ldots\ldots\ldots,0}_{2\mu -n\mbox{~terms}},C_{n-2}^{0}r_k^0,\ldots,(C_{n-2}^{\nu }+C_{n-2}^{\nu -1})r_k^\nu ,\ldots\Big)^T
\qquad\qquad\qquad\qquad\qquad\qquad\qquad\quad\\[25pt]
\qquad\qquad\qquad\qquad\mbox{if}\,~\mu > \left\lceil \dfrac{n-2}{2}\right\rceil\mbox{~and~}0\leq \nu \leq 2\,n-2\,\mu -3,
\end{array}
\right.
\]
where $C_{n-2}^\nu =0$ for $\nu <0$ and $\nu >n-2$, and $\lceil \gamma \rceil$ denotes the smallest integer that is not less than the rational number $\gamma$. For any positive integer $l$, denote by ${\rm co}_l$ the $l$th column and by ${\rm ro}_l$ the $l$th row of this matrix. Then
\begin{align*}
 M_n(r_k) & \xlongequal[\mu =n-3,\ldots,1]{{\rm co}_\mu +{\rm co}_{\mu +1}\cdot r_k^2}\\\\ &\left|
\begin{array}{cccccccc}
\sum\limits_{\mu =0}^{n-2}C_{n-2}^{\mu }r_k^{n-2} & \sum\limits_{\mu =0}^{n-4}C_{n-2}^{\mu }r_k^{n-4} & \cdots \!&\! \cdots & \sum\limits_{\mu =1}^{n-2\nu }C_{n-2}^{\mu }r_k^{n-2\nu } & \cdots \!&\! \cdots\\[12pt]
\sum\limits_{\mu =0}^{n-2}C_{n-2}^{\mu }r_k^{n-1} & \sum\limits_{\mu =0}^{n-3}C_{n-2}^{\mu }r_k^{n-3} & \cdots \!&\! \cdots & \sum\limits_{\mu =1}^{n-2\nu +1}C_{n-2}^{\mu }r_k^{n-2\nu +1} & \cdots \!&\! \cdots\\[12pt]
\sum\limits_{\mu =0}^{n-2}C_{n-2}^{\mu }r_k^{n} &\sum\limits_{\mu =0}^{n-2}C_{n-2}^{\mu }r_k^{n-2} & \cdots \!&\! \cdots &\sum\limits_{\mu =1}^{n-2\nu +2}C_{n-2}^{\mu }r_k^{n-2\nu +2} & \cdots \!&\! \cdots\\[12pt]
\sum\limits_{\mu =0}^{n-2}C_{n-2}^{\mu }r_k^{n+1} & \sum\limits_{\mu =0}^{n-2}C_{n-2}^{\mu }r_k^{n-1} & \cdots \!&\! \cdots
&\sum\limits_{\mu =1}^{n-2\nu +3}C_{n-2}^{\mu }r_k^{n-2\nu +3} &\cdots \!&\! \cdots\\
\vdots & \vdots & \ddots \!&\! & \vdots &\ddots \!&\!\\[-11pt]
\vdots & \vdots & \!&\!\hspace{-5pt}\raisebox{0.1cm}{\mbox{$\ddots$}} & \vdots &\!&\!\hspace{-5pt}\raisebox{0.1cm}{\mbox{$\ddots$}}
\end{array}
\right|
\end{align*}
\begin{align*}
\xlongequal[\mu =n-2,\ldots,2]{{\rm ro}_{\mu }-{\rm ro}_{\mu -1}\cdot r_k}
&\left|
\begin{array}{cccccccc}
\sum\limits_{\mu =0}^{n-2}C_{n-2}^{\mu }r_k^{n-2} & \sum\limits_{\mu =0}^{n-4}C_{n-2}^{\mu }r_k^{n-4} & \sum\limits_{\mu =0}^{n-6}C_{n-2}^{n-6}r_k^{n-6} & \cdots \!&\! \cdots & \sum\limits_{\mu =0}^{n-2\nu }C_{n-2}^{n-2\nu }r_k^{n-2\nu } & \cdots \!&\! \cdots\\[12pt]
0 & C_{n-2}^{n-3}r_k^{n-3} & C_{n-2}^{n-5}r_k^{n-5} & \cdots \!&\! \cdots & C_{n-2}^{n-2\nu +1}r_k^{n-2\nu +1} & \cdots \!&\! \cdots\\[12pt]
0 &C_{n-2}^{n-2}r_k^{n-2} & C_{n-2}^{n-4}r_k^{n-4} & \cdots \!&\! \cdots & C_{n-2}^{n-2\nu +2}r_k^{n-2\nu +2} & \cdots \!&\! \cdots\\[12pt]
0 &0 & C_{n-2}^{n-3}r_k^{n-3} & \cdots \!&\! \cdots & C_{n-2}^{n-2\nu +3}r_k^{n-2\nu +3}  & \cdots \!&\! \cdots\\[12pt]
0 & 0 & C_{n-2}^{n-2}r_k^{n-2} & \cdots \!&\! \cdots & C_{n-2}^{n-2\nu +4}r_k^{n-2\nu +4} & \cdots \!&\! \cdots\\
\vdots & \vdots & \vdots & \ddots \!&\! & \vdots& \ddots \!&\! \\[-11pt]
\vdots & \vdots & \vdots & \!&\! \hspace{-4pt}\raisebox{0.1cm}{\mbox{$\ddots$}} &\vdots & \!&\!\hspace{-4pt}\raisebox{0.1cm}{\mbox{$\ddots$}}
\end{array}
\right|\\ \\
=&\sum_{\mu =0}^{n-2}C_{n-2}^{\mu }r_k^{n-2}\cdot M_{n-1}(r_k)=(2\,r_k)^{n-2}\cdot M_{n-1}(r_k).
\end{align*}

Since $M_3(r_k)=2\,r_k$, it is easy to verify that
\[M_n(r_k)=(2\,r_k)^{\frac{(n-1)(n-2)}{2}}\]
by induction.
\end{proof}

\begin{proof}[Proof of Theorem 2]
By Lemma \ref{lem:division}, \[\left[\prod\nolimits_{i< j\neq k}^n{(2\,r_k-r_i-r_j)}\right]~\bigg|~
H(r_k).\]
Hence there exists a polynomial $P=P(r_1,\ldots,r_n)$ such that
\[
H(r_k)=P\,\prod_{i< j\neq k}^n{(2\,r_k-r_i-r_j)}.
\]
Observe that both of the leading terms of $\prod_{i< j\neq k}^n{(2\,r_k-r_i-r_j)}$ and $H$ with respect to $r_k$ is equal to $(2\,r_k)^{\frac{(n-1)(n-2)}{2}}$. This implies that $P=1$, so that
\[H(r_k)=\prod_{i< j\neq k}^n{(2\,r_k-r_i-r_j)}.\]
Therefore,
\[\res(f, H, x)=\prod_{k=1}^nH(r_k)=\prod_{i< j\neq k}(2\,r_k-r_i-r_j)=D_2.\]

\vspace{-0.5cm}
\end{proof}

%====================================================================================
\section{Irreducibility and Degree of the Second Discriminant}
\label{sec:Properties}
%====================================================================================

Using Theorem \ref{thm:D2}, one can easily verify that
\begin{enumerate}[(1)]
\item for $n=3$, $D_2=-2\,a_2^3+9\,a_1a_2-27\,a_0$;
\item for $n=4$, $D_2$ is an irreducible polynomial of total degree $9$, and more explicitly:
\begin{align*}
    \!\!\!\!\!\!\!\!D_2\,=\,&216\,a_0a_3^8-72\,a_1a_2a_3^7+16\,a_2^3a_3^6-2304\,a_0a_2a_3^6+72\,a_1^2a_3^6+672\,a_1a_2^2a_3^5-144\,a_2^4
a_3^4\\
&+5310\,a_0a_1a_3^5+7446\,a_0a_2
^2a_3^4-2346\,a_1^2a_2a_3^4-1278\,a_1
a_2^3a_3^3+324\,a_2^5a_3^2\\
&-9675\,a_
0^2a_3^4\!-\!28950\,a_0a_1a_2a_3^3-6804
\,a_0a_2^3a_3^2+1658\,a_1^3a_3^3+
9423\,a_1^2a_2^2a_3^2\!-\!1296\,a_1a_2^
4a_3\\
&+\!51600\,a_0^2a_2a_3^2\!+\!31890\,a_0a_
1^2a_3^2\!+\!19440\,a_0a_1a_2^2a_3\!+\!1296\,
a_0a_2^4\!-\!17262\,a_1^3a_2a_3\!+\!972\,a_1
^2a_2^3\\
&-120000\,a_0^2a_1a_3-28800\,a_0
^2a_2^2-5040\,a_0a_1^2a_2+9261\,a_1^4
+160000\,a_0^3;
\end{align*}
\item for $n=5$, $D_2$ is an irreducible polynomial of total degree $18$, consisting of 521 terms, in $a_0,\ldots,a_4$.
\end{enumerate}

In what follows, we prove that $D_2$ is an irreducible polynomial of total degree $3\,(n-1)(n-2)/2$ in $a_0,\ldots,a_{n-1}$ for any $n\geq3$. For this purpose, let $s_i=\sum r_{k_1}\cdots r_{k_i}$ be the sum of all the possible, distinct products of $i$ elements taken from $r_1,\ldots, r_n$ for $i=1,\ldots,n$. The sum $s_i$ of products is called the {\em elementary symmetric polynomial} of degree $i$ in $r_1,\ldots, r_n$. It is easy to show that the Vieta formula $a_{n-i}=(-1)^is_i$ holds for $i=1,\ldots,n$.

\begin{proposition}\label{prop:irreducibility}
Let $\bm{a}=(a_0, a_1,\ldots,a_{n-1})$ be the coefficients and $D_2$ be the second discriminant of a monic univariate polynomial $f$.
Then $D_2(a_0,\ldots,a_{n-1})\in\mathbb{Q}[a_0,\ldots,a_{n-1}]$ is irreducible over $\mathbb{Q}$.
\end{proposition}

\begin{proof}
Let $P\in \mathbb{Q}[a_0,\ldots,a_{n-1}]$ be a nonconstant irreducible polynomial and suppose that $P\mid D_2$. We show that $D_2\mid P$.

Substituting Vieta's formula $a_{n-i}=(-1)^i\sum r_{k_1}\cdots
r_{k_i}$ into $P$ and $D_2$, we obtain two symmetric polynomials $\bar{P}$ and $\bar{D}_2$ in $\mathbb{Q}[r_1,\ldots,r_n]$, respectively. Then $\bar{P}=0$ is equivalent to $P=0$, and so is $\bar{D}_2$ to $D_2$. Since $P$ is nonconstant, so is $\bar{P}$. As $P\mid D_2$, $\bar{P}\mid \bar{D}_2$; so $\bar{P}$ contains at least one irreducible factor of $\bar{D}_2$, say $2\,r_1-r_2-r_3$. Therefore, every $2\,r_k-r_i-r_j$ is a factor of $\bar{P}$ because $\bar{P}$ is symmetric with respect to $r_1,\ldots,r_n$. It follows that $\bar{D}_2\mid \bar{P}$. Hence $\bar{D}_2$ and $\bar{P}$ differ only by a nonzero constant factor, and so do $D_2$ and $P$. Therefore, $D_2\mid P$ and thus $D_2$ is irreducible over $\mathbb{Q}$.
\end{proof}

For simplicity, we often write $\bm{a}$ for $(a_0, a_1,\ldots,a_{n-1})$ and $\deg(F,\bm{a})$ for the total degree of $F$ in $\bm{a}$.

\begin{proposition}\label{prop:degD2}
Let $\bm{a}=(a_0, a_1,\ldots,a_{n-1})$ be the coefficients and $D_2$ be the second discriminant of a monic univariate polynomial $f$ of degree $n$.  Then the total degree $\deg(D_2,\bm{a})$ of $D_2$ in $\bm{a}$ is $3\,(n-1)(n-2)/2$.
\end{proposition}

\begin{proof}
Set $B_0=D_2$. For $i=1,\ldots, n$, let $C_i$ be the homogeneous part of $B_{i-1}$ of the highest total degree in $\bm{a}^+=(a_0,\ldots,a_{n-1},r_1)$ and let $B_i$ be obtained from $C_i$ by substituting Vieta's formula $a_{n-i}=(-1)^i\sum r_{k_1}\cdots r_{k_i}=U_{n-i}r_1+V_{n-i}$, where $U_{n-i}\neq0$ and $\deg(U_{n-i},r_1)=\deg(V_{n-i},r_1)=0$. Then
\begin{align*}
C_i=S_{n-i}a_{n-i}^{N_i}+T_{n-i}&=S_{n-i}(U_{n-i}r_1+V_{n-i})^{N_i}+T_{n-i}=B_i,\\
S_{n-i}U_{n-i}^{N_{i}}r_1^{N_{i}}&=C_{i+1},
\end{align*}
where $N_i=\deg(C_i,a_{n-i})$, $S_{n-i}$ is the leading coefficient of $C_i$ with respect to $a_{n-i}$, and $S_{n-i}U_{n-i}\neq0$.
Therefore, the total degrees of $C_i$, $B_i$, $C_{i+1}$ in $\bm{a}^+$ remain the same for $i=1,\ldots,n$, so $\deg(C_1,\bm{a}^+)=\deg(C_n,\bm{a}^+)$. Note that $C_n$ is the leading term of $D_2$, expressed in terms of the roots $r_1,\ldots,r_n$ as in \eqref{eq:D2}, with respect to $r_1$ and $\deg(C_n, \bm{a}^+)=\deg(C_n, r_1)=3\,(n-1)(n-2)/2$. Thus $\deg(D_2,\bm{a})=\deg(C_1,\bm{a}^+)=3\,(n-1)(n-2)/2$ and the proposition is proved.
\end{proof}

\section{The Second Discriminant with Resultants}\label{sec:resultant}

The following three polynomials will play a significant role in this and later sections:
\begin{align}
& f_1(x,y)=\dfrac{f(y)-f(x)}{y-x},\label{eq:f1}\\[-2pt]
& f_2(x,y)=\dfrac{f\left(\dfrac{x+y}{2}\right)-f(x)}{\dfrac{y-x}{2}},\label{eq:f2}\\[-2pt]
&f_3(x,y)=\dfrac{f(y)-2\,f\left(\dfrac{x+y}{2}\right)+f(x)}{\dfrac{(y-x)^2}{2}}.\label{eq:f3}
\end{align}
The rational functions on the right-hand side of the above equalities can all be simplified to polynomials in $x$ and $y$.

\begin{proposition}\label{prop:resD1D2}
Let $f$ be a univariate polynomial and $f_1,f_2$ be as in \eqref{eq:f1} and \eqref{eq:f2}. Then
$$\res(f, \res(f_1, f_2, y), x)=0 ~~\mbox{if and only if}~~ D_1D_2=0.$$
\end{proposition}

\begin{proof}
($\Longleftarrow$) Let
\[R_1(x)=\res(f_1, f_2, y), \quad R_2=\res(f, R_1, x).\]
We want to show that, if $D_1D_2=0$, then there exist $r_i$ and $r_j$ such that
\begin{equation}\label{eq:F12}
f(r_i)=f(r_j)=0,\quad f_1(r_i,r_j)=0,\quad f_2(r_i, r_j)=0.
\end{equation}
For this purpose, first suppose that $D_1=0$. Then there exist $r_i=r_j, i\neq j$, such that $f(r_i)=f(r_j)=0$ and $f'(r_i)=f'(r_j)=0$ (where $'$ is the derivation operator).
Note that
\[f_1(x,y)=\sum_{k=0}^{n-1}\frac{f^{(k+1)}(x)}{(k+1)!}(y-x)^k,\quad f_2(x,y)=\sum_{k=0}^{n-1}\frac{f^{(k+1)}(x)}{(k+1)!}\left(\frac{y-x}{2}\right)^k.\]
Substitution of $x=r_i$ and $y=r_j$ into the above expressions shows that \eqref{eq:F12} holds in this case.

Now suppose that $D_2=0$ and $D_1\neq0$. Then there exist $r_i\neq r_j$ such that $f(r_i)=f(r_j)=0$ and $f\left(\frac{r_i+r_j}{2}\right)=0$. It follows that
\[f_1(r_i,r_j)=\dfrac{f(r_j)-f(r_i)}{r_j-r_i}=0,\quad f_2(r_i,r_j)=\dfrac{f\left(\dfrac{r_i+r_j}{2}\right)-f(r_i)}{\dfrac{r_j-r_i}{2}}=0.\]
Thus \eqref{eq:F12} holds as well.

In any case, $f_1(r_i,y)$ and $f_2(r_i,y)$ have a common zero $r_j$ for $y$. Therefore,
\[R_1(r_i)=\res(f_1(r_i, y), f_2(r_i,y),y)=0.\]
Hence $f(x)$ and $R_1(x)$ have a common root $r_i$ for $x$. This implies that $R_2=0$.

($\Longrightarrow$) $\res(f, \res(f_1, f_2, y), x)=0$ implies that there exist $r_i$ and $r_j$, $i<j$, such that
\[f(r_i)=0,\quad f_1(r_i, r_j)=f_2(r_i, r_j)=0.\]
Thus $f(r_j)=f_1(r_i,r_j)(r_j-r_i)+f(r_i)=0$, which indicates that $r_j$ is also a root of $f$.

If $r_j=r_i$, then $f(x)$ has a multiple root and thus $D_1=0$.
Otherwise,
\[f_2(r_i,r_j)=2\left[f\left(\frac{r_i+r_j}{2}\right)-f(r_i)\right]\Big/(r_j-r_i)=0\]
implies that $f\left(\frac{r_i+r_j}{2}\right)=0$, so $f$ has three roots, which form a symmetric triple. Therefore $D_2=0$.
\end{proof}

Using similar ideas, we can prove the following proposition, which shows
how to construct $D_2$ via resultant computation twice.

\begin{proposition}\label{prop:resD2}
Let $f$ be a univariate polynomial and $f_1,f_3$ be as in \eqref{eq:f1} and \eqref{eq:f3}. Then
$$\res(f, \res(f_1, f_3, y), x)=0~~ \mbox{if and only if} ~~D_2=0.$$
\end{proposition}

\begin{proof}
($\Longleftarrow$) Let
\[F(x)=\res(f_1, f_3, y),\quad E=\res(f, F, x).\]
We show that, if $D_2=0$, then there exist $r_i$ and $r_j$ such that
\begin{equation}\label{eq:F13}
f(r_i)=f(r_j)=0,\quad f_1(r_i,r_j)=0,\quad f_3(r_i, r_j)=0.
\end{equation}

First suppose that there exist $r_i=r_j=r_k$, $i< j\neq k$, such that $f(r_i)=f(r_j)=f(r_k)=0$. Then $f'(r_i)=f''(r_i)=0$.
Note that
\begin{align*}
f_1(x,y)&=\sum_{k=0}^{n-1}\frac{f^{(k+1)}(x)}{(k+1)!}(y-x)^k,\\
f_3(x,y)&=\dfrac{f_1-f_2}{\dfrac{y-x}{2}}=\dfrac{\sum\limits_{k=0}^{n-1}\dfrac{f^{(k+1)}(x)}{(k+1)!}(y-x)^k-\sum\limits_{k=0}^{n-1}\dfrac{f^{(k+1)}(x)}{(k+1)!}\left(\dfrac{y-x}{2}\right)^k}{\dfrac{y-x}{2}}\\[3pt]
&=\sum_{k=0}^{n-2}\left(2-\dfrac{1}{2^{k}}\right)\dfrac{f^{(k+2)}(x)}{(k+2)!}\left(y-x\right)^k.
\end{align*}
Substitution of $x=r_i$ and $y=r_j$ into the above expressions shows that \eqref{eq:F13} holds in this case.

Suppose otherwise that there exist $r_i\neq r_j$ such that $f(r_i)=f(r_j)=0$ and $f\left({(r_i+r_j)}/{2}\right)=0$. Then it follows from $D_2=0$ that
\[f_1(r_i,r_j)=\dfrac{f(r_j)-f(r_i)}{r_j-r_i}=0,\quad f_3(r_i,r_j)=\dfrac{f(r_j)-2\,f\left(\dfrac{r_i+r_j}{2}\right)+f(r_i)}{\dfrac{(r_j-r_i)^2}{2}}=0,\]
so \eqref{eq:F13} holds as well.

In any case, $f_1(r_i,y)$ and $f_3(r_i,y)$ have a common zero $r_j$ for $y$. Therefore,
\[F(r_i)=\res(f_1(r_i, y), f_3(r_i,y),y)=0.\]
Hence $f(x)$ and $F(x)$ have a common root $r_i$ for $x$. This implies that $E=0$.

($\Longrightarrow$) $\res(f, \res(f_1, f_3, y), x)=0$ implies that there exist $r_i$ and $r_j$, $i< j$, such that
\[f(r_i)=0,\quad f_1(r_i, r_j)=f_3(r_i, r_j)=0.\]
Moreover, $f(r_i)=0$ and $f_1(r_i,r_j)=0$ imply that $f(r_j)=0$.

Consider first the case when $r_i=r_j$. In this case, $D_1=0$ and thus $f'(r_i)=0$. The following calculation shows that $f''(r_i)=0$:
\begin{align*}
f''(r_i)&=~\lim_{x\rightarrow r_i}\dfrac{f'\left(\dfrac{x+r_i}{2}\right)-f'(r_i)}{\dfrac{x-r_i}{2}}\\
&=~\lim_{x\rightarrow r_i}\dfrac{\dfrac{f(x)-f\left(\dfrac{x+r_i}{2}\right)}{\dfrac{x-r_i}{2}}-\dfrac{f\left(\dfrac{x+r_i}{2}\right)-f(r_i)}{\dfrac{x-r_i}{2}}}{\dfrac{x-r_i}{2}}\\
&=~2\,\lim_{x\rightarrow r_i}\dfrac{f(x)+f(r_i)-2\,f\left(\dfrac{x+r_i}{2}\right)}{\dfrac{(x-r_i)^2}{2}}\\[2pt]
&=~2\lim_{x\rightarrow r_i}f_3(x,r_i)=~2\,f_3(r_i,r_j)=0.
\end{align*}
Therefore, there exists an $r_k$ such that $k\neq i$, $k\neq j$ and $r_k=r_i=r_j$, which implies that $2\,r_k-r_i-r_j=0$. Hence $D_2=0$.

Now consider the case when $r_j\neq r_i$. In this case,
\[f_3(r_i,r_j)=\left[f(r_i)+f(r_j)-2\,f\left(\dfrac{r_i+r_j}{2}\right)\right]\bigg/\dfrac{(r_j-r_i)^2}{2}=0\]
implies that $f\left({(r_i+r_j)}/{2}\right)=0$, so $x={(r_i+r_j)}/{2}$ is a root of $f$. Therefore $D_2=0$.
\end{proof}

Since $D_2$ is irreducible over ${\Bbb Q}$, there exist a positive integer $q$ and a nonzero constant $c\in{\Bbb Q}$ such that
\[D_2^q=c\cdot\res(f, \res(f_1, f_3, y), x).\]
In what follows, we prove that $q=2$. For simplicity, we write $F$ for $\res(f_1, f_3, y)$ and $E$ for $\res(f, F, x)$.

\begin{theorem}\label{thm:qIs2}
Let $f$ be a univariate polynomial and $f_1,f_3$ be as in \eqref{eq:f1} and  \eqref{eq:f3}, and let
$F(x)=\res(f_1, f_3, y)$ and $E=\res(f, F, x)$.
Then $E=c\, D_2^2$, where $c$ is a nonzero rational number.
\end{theorem}

The proof of this theorem requires Lemmas \ref{lem:Fr1} and \ref{lem:deg_res_ff1f3}, of which the latter shows that $\deg(E, \bm{a})\leq 3\,(n-1)(n-2)+2\,(n-2)$.

\begin{lemma}\label{lem:Fr1}
Let $r_1,\ldots,r_n$ be the $n$ roots of a univariate polynomial $f$ and
$F(x)$ be the resultant of $f_1$ in \eqref{eq:f1} and $f_3$ in \eqref{eq:f3} with respect to $y$.
Then for any $k,j$ with $1<k\neq j$, $r_1-2\,r_k+r_j$ divides $F(r_1)$.
\end{lemma}

\begin{proof}
 It suffices to show that $F(r_1)=0$ when $r_1=2\,r_k-r_j$ for any fixed $k,j$ satisfying $1<k\neq j$.

According to the theory of resultants \cite[pp. 228]{M1993A}, there exist polynomials $A_1(x,y)$ and $A_3(x,y)$ such that
\[F(x)=A_1(x,y)f_1(x,y)+A_3(x,y)f_3(x,y).\]
Suppose that $r_j\neq r_1$. Since $f(r_1)=f(r_k)=f(r_j)=0$, substitution of $x=r_1$ and $y=r_j$ into $f_1$ and $f_3$ yields
\begin{align*}
f_1(r_1,r_j)&=\dfrac{f(r_j)-f(r_1)}{r_j-r_1}=0,\\
f_3(r_1,r_j)&=\dfrac{f(r_j)-2\,f\left(\dfrac{r_1+r_j}{2}\right)+f(r_1)}
{\dfrac{(r_j-r_1)^2}{2}}=\dfrac{f(r_j)-2\,f(r_k)+f(r_1)}
{\dfrac{(r_j-r_1)^2}{2}}=0.
\end{align*}
Suppose otherwise that $r_j=r_1$. Then $r_k=(r_1+r_j)/2=r_1$, which implies that $x=r_1$ is a root of $f$ with multiplicity greater than $2$. Thus $f(r_1)=f'(r_1)=f''(r_1)=0$. It follows that
\begin{align*}
f_1(r_1,r_j)&=\sum_{k=0}^{n-1}\frac{f^{(k+1)}(r_1)}{(k+1)!}(r_j-r_1)^k=f'(r_1)=0,\\
f_3(r_1,r_j)&=\sum_{k=0}^{n-2}\left(2-\dfrac{1}{2^{k}}\right)\dfrac{f^{(k+2)}(r_1)}{(k+2)!}\left(r_j-r_1\right)^k=\dfrac{f''(r_1)}{2!}=0.
\end{align*}
Hence, in both cases we have $f_1(r_1,r_j)=f_3(r_1,r_j)=0$. Therefore
\[F(r_1)=A_1(r_1,r_j)f_1(r_1,r_j)+A_3(r_1,r_j)f_3(r_1,r_j)=0,\]
so $r_1-2\,r_k+r_j$ divides $F(r_1)$.
\end{proof}

\begin{proof}[Proof of Theorem \ref{thm:qIs2}]
By Lemma \ref{lem:Fr1}, $(r_1-2\,r_k+r_j)\mid F(r_1)$ for arbitrarily chosen $k,j$ with $1<k\neq j$. Hence
\[\prod_{1<k\neq j}{(r_1-2\,r_k+r_j)}\mid F(r_1).\]
It follows from the theory of resultants \cite[p.\,398]{GKZ1994D} that
\[E=\prod_{i=1}^n F(r_i)=\prod_{i< j\neq k}{(r_i-2\,r_k+r_j)^2}\cdot K=D_2^2\cdot K\]
for some polynomial $K$ in $r_1,\ldots,r_n$. By Proposition~\ref{prop:resD2}, there exist a nonzero constant $c$ and an integer $q\geq2$ such that $E=c\, D_2^q$.

On the other hand, by Lemma \ref{lem:deg_res_ff1f3}, $\deg(E,\bm{a})\leq 3\,(n-1)(n-2)+2\,(n-2)$; by Proposition~\ref{prop:degD2}, $\deg(D_2,\bm{a})= {3\,(n-1)(n-2)}/{2}$. Under these constraints, the only possibility for $E=c\, D_2^q$ to hold is that $K$ is a constant and $q=2$.
\end{proof}

\section{The Second Discriminant with Ideals}\label{sec:ideals}

In searching for explicit representations of $D_2$ in terms of the coefficients of $f$, we have discovered the amazingly structured matrix $M$ formed with the derivatives of $f$ shown in \eqref{matM}. In what follows, we establish an inherent connection between the $(n-2)$th leading principal minor $H$ of $M$ and $\res(f_1, f_3, y)$, which reveals the hidden mystery for the structure of $M$.

Let $\langle f_1,\ldots,f_m\rangle$ denote the ideal generated by $f_1,\ldots,f_m$ in a ring of polynomials. The polynomials $f_1,\ldots,f_m$ are called the generators of the ideal.

\begin{lemma}\label{lem:equiv_ideal}
Let $f$ be a univariate polynomial and $f_1$, $f_2$, $f_3$ be as in \eqref{eq:f1}--\eqref{eq:f3}. Then
\begin{align*}
~\qquad\qquad&\!\!\!\!\!\!\!\!\!\!\!\!\!\!\!\!\!\!\!\!\!\!\!\!\!\!\!\!\!\!\left\langle f(x),f(y), f\left(\frac{x+y}{2}\right),w(x-y)-1\right\rangle\\
&=\left\langle f(x),f(y)-f(x), f\left(\frac{x+y}{2}\right)-f(x),w(x-y)-1\right\rangle\\[1pt]
&=\left\langle f(x),f_1(x,y),f_2(x,y),w(x-y)-1\right\rangle\\[4pt]
&=\left\langle f(x),f_1(x,y),f_3(x, y),w(x-y)-1\right\rangle,
\end{align*}
where $w$ is a new indeterminate.
\end{lemma}

\begin{proof}
Let the four ideals in the above identity be denoted successively by $\ideal{I}_1, \ldots, \ideal{I}_4$. It is obvious that $\ideal{I}_1=\ideal{I}_2$. We only need to show that $\ideal{I}_2=\ideal{I}_3$ and $\ideal{I}_3=\ideal{I}_4$.
\begin{enumerate}[(1)]
\item Since $f(y)-f(x)=f_1\cdot(y-x)$ and $f\left(\dfrac{x+y}{2}\right)-f(x)=f_2\cdot\dfrac{y-x}{2}$, we have $\ideal{I}_2\subset \ideal{I}_3$. On the other hand,
\begin{align*}
f_1&=-w\left[f(y)-f(x)\right]-[w(x-y)-1]\frac{f(y)-f(x)}{y-x},\\
f_2&=-2\,w\left[f\left(\dfrac{x+y}{2}\right)-f(x)\right]-2\,[w(x-y)-1]\frac{f\left(\dfrac{x+y}{2}\right)-f(x)}{y-x},
\end{align*}
so $\ideal{I}_3\subset \ideal{I}_2$.
\item $\ideal{I}_3\subset \ideal{I}_4$ follows from $f_2=f_1+\dfrac{x-y}{2}\cdot f_3$. $\ideal{I}_4\subset \ideal{I}_3$ can be easily deduced from $f_3=-2\,wf_1+2\,wf_2-f_3[w(x-y)-1]$.
\end{enumerate}
Therefore $\ideal{I}_1=\ideal{I}_2=\ideal{I}_3=\ideal{I}_4$.
\end{proof}

\begin{lemma}\label{lem:resH}
Let $f_1$, $f_3$ be as in \eqref{eq:f1} and \eqref{eq:f3}, and let
\begin{equation}\label{eq:g13}
g_1(x,y)=f_1(x-y,x+y), \quad g_3(x,y)=f_3(x-y,x+y),\quad G=\res(g_1,g_3,y).
\end{equation}
Then
$G=H^2$, where $H$ is as in Theorem~\ref{thm:D2}.
\end{lemma}

\begin{proof}
For any rational number $\gamma$, denote by $\lfloor \gamma \rfloor$ the biggest integer that is not greater than $\gamma$. Taking Taylor expansion for $g_1$ and $g_3$ at $x$, we have
\begin{align*}
g_1(x,y)&=\dfrac{f(x+y)-f(x-y)}{2\,y}\\
&=\dfrac{f(x+y)-f(x)}{2\,y}-\dfrac{f(x-y)-f(x)}{2\,y}\\
&=\dfrac{1}{2}\sum_{k=1}^n\left[1-(-1)^k\right]\dfrac{f^{(k)}(x)}{k!}y^{k-1}\\
&=\sum_{k=0}^{\lfloor\frac{n-1}{2}\rfloor}\dfrac{f^{(2\,k+1)}(x)}{(2\,k+1)!}y^{2\,k}
\end{align*}
and
\begin{align*}
g_3(x,y)&=\dfrac{f(x+y)+f(x-y)-2\,f(x)}{2\,y^2}\\
&=\dfrac{1}{2\,y}\cdot\left[\dfrac{f(x+y)-f(x)}{y}+\dfrac{f(x-y)-f(x)}{y}\right]\\
&=\dfrac{1}{2}\sum_{k=2}^n\left[1+(-1)^k\right]\dfrac{f^{(k)}(x)}{k!}y^{k-2}\\
&=\sum_{k=0}^{\lfloor\frac{n}{2}\rfloor-1}\dfrac{f^{(2\,k+2)}(x)}{(2\,k+2)!}y^{2\,k}.
\end{align*}
Let $g^*_1$ and $g^*_3$ be obtained from $g_1$ and $g_3$ by replacing $y^2$ with  $z$.
Then
\[\res(g_1^*,g_3^*,z)=
\begin{array}{c@{\hspace{-5pt}}l}
\left|\begin{array}{ccccc}
\frac{f^{(n-1)}}{(n-1)!}&\frac{f^{(n-3)}}{(n-3)!}&\frac{f^{(n-5)}}{(n-5)!}&\cdots\!&\!\cdots\\ [6pt]
0&\frac{f^{(n-1)}}{(n-1)!}&\frac{f^{(n-3)}}{(n-3)!}&\cdots\!&\!\cdots\\ [6pt]
\vdots&\vdots&\vdots&\ddots\!&\!\\[-10pt]
\vdots&\vdots&\vdots&\!&\!\hspace{-4pt}\raisebox{0.1cm}{\mbox{$\ddots$}}\\ [6pt]
\frac{f^{(n)}}{n!}&\frac{f^{(n-2)}}{(n-2)!}&\frac{f^{(n-4)}}{(n-4)!}&\cdots\!&\!\cdots\\ [6pt]
0&\frac{f^{(n)}}{n!}&\frac{f^{(n-2)}}{(n-2)!}&\cdots\!&\!\cdots\\ [6pt]
\vdots&\vdots&\vdots&\ddots\!&\!\\[-10pt]
\vdots&\vdots&\vdots&\!&\!\hspace{-4pt}\raisebox{0.1cm}{\mbox{$\ddots$}}
\end{array}\right|
&
\begin{array}{l}\left.\rule{0mm}{13mm}\right\}\left\lfloor{\dfrac{n}{2}}\right\rfloor-1\\ [6pt]
\\\left.\rule{0mm}{13mm}\right\}\left\lfloor{\dfrac{n-1}{2}}\right\rfloor
\end{array}
\end{array}
=\pm\, H.\]
Therefore,
\[G=\res(g_1,g_3,y)=[\res(g_1^*(x,z),g_3^*(x,z),z)]^2=H^2.\]

\vspace{-0.5cm}
\end{proof}

\begin{corollary}
Let $f$ be a monic univariate polynomial of degree $n$ with coefficients $a_0,\ldots,a_{n-1}$ and $D_2$ be the second discriminant of $f$. Then
$$D_2\in\left\langle f(x), g_1(x,y),g_3(x,y)\right\rangle\cap\mathbb{Q}[a_0,\ldots,a_{n-1}]$$
where $g_1$ and $g_3$ are as in \eqref{eq:g13}.
\end{corollary}

\begin{proof}
Let ${\cal K}=\mathbb{Q}[a_0,\ldots,a_{n-1}]$.
By Lemma \ref{lem:resH},
\[\res(f,G,x)=\res(f, H^2,x)=D_2^2.\]
According to the theory of resultants, there exist $A_1(x,z),A_2(x,z)\in{\cal K}[x,z]$ such that
\[H=A_1(x,z)g_1^*(x,z)+A_2(x,z)g_3^*(x,z).\]
Similarly, there exist $B_1(x),B_2(x)\in{\cal K}[x]$ such that
\begin{align*}
D_2&=\res(f,H,x)=B_1(x)f(x)+B_2(x)H\\
&=B_1(x)f(x)+B_2(x)[A_1(x,z)g_1^*(x,z)+A_2(x,t)g_3^*(x,z)].
\end{align*}
Substituting $z=y^2$, one gets
\begin{align*}
D_2&=B_1(x)f(x)+A_1(x,y^2)B_2(x)g_1(x,y)+A_2(x,y^2)B_2(x)g_3(x,y)\\
&\in\langle f, g_1,g_3\rangle\cap{\cal K}.
\end{align*}
The corollary is proved.
\end{proof}

\begin{theorem}\label{thm:D2in}
Let $f$ be a monic univariate polynomial of degree $n$ with coefficients $a_0,\ldots,a_{n-1}$ and $D_2$ be the second discriminant of $f$. Then
$$D_2\in\left\langle f(x), f_1(x,y),f_3(x,y)\right\rangle\cap\mathbb{Q}[a_0,\ldots,a_{n-1}]$$
where $f_1$ and $f_3$ are as in \eqref{eq:f1} and \eqref{eq:f3}.
\end{theorem}

\begin{proof}
Replace $x$ and $y$ in $f(x)$, $f_1(x-y,x+y)$, $f_3(x-y,x+y)$ by ${(Y+X)}/{2}$ and ${(Y-X)}/{2}$, respectively.
Since
$$D_2\in \langle f(x), f_1(x-y,x+y),f_3(x-y,x+y)\rangle$$
and $D_2$ does not involve $x$ and $y$,
\[D_2\in \left\langle f\left(\dfrac{X+Y}{2}\right),f_1(X,Y), f_3(X,Y)\right\rangle.\]

Furthermore, from \[f\left(\dfrac{X+Y}{2}\right)=-\dfrac{1}{4}(Y-X)^2f_3(X,Y)+\dfrac{1}{2}(Y-X)f_1(X,Y)+f(X),\] one can deduce
\[\langle f(X),f_1(X,Y), f_3(X,Y)\rangle=\left\langle f\left(\dfrac{X+Y}{2}\right),f_1(X,Y), f_3(X,Y)\right\rangle.\]
Therefore,
\[D_2\in\langle f(X),f_1(X,Y), f_3(X,Y)\rangle.\]
Substitution of $X=x$ and $Y=y$ back to the above expression, we have
\[D_2\in\langle f(x),f_1(x,y), f_3(x,y)\rangle\]
The proof is complete.
\end{proof}

\begin{corollary}\label{cor:D2in}
Let $f$ be a monic univariate polynomial of degree $n$ with coefficients $a_0,\ldots,a_{n-1}$ and $D_2$ be the second discriminant of $f$. Then
\[D_2\in\left\langle f(x),f(y),f\left(\frac{x+y}{2}\right),w(x-y)-1\right\rangle\cap\mathbb{Q}[a_0,\ldots,a_{n-1}]\]
where $w$ is a new indeterminate.
\end{corollary}

\begin{proof}
It follows from Lemma \ref{lem:equiv_ideal} and Theorem \ref{thm:D2in}.
\end{proof}

\begin{proposition}\label{propD2}
Let $f$ be a monic univariate polynomial of degree $n$ with coefficients $a_0,\ldots,a_{n-1}$ and $D_2$ be the second discriminant of $f$. Then
\[\langle D_2\rangle=\left\langle f(x),f(y), f\left(\frac{x+y}{2}\right),w(x-y)-1\right\rangle\cap \mathbb{Q}[a_0,\ldots,a_{n-1}]\]
where $w$ is a new indeterminate
\end{proposition}

\begin{proof}
Let $\ideal{I}_1$ and $\ideal{I}_4$ be as in the proof of Lemma \ref{lem:equiv_ideal}, which implies that
\[\ideal{I}_1\cap {\cal K}=\ideal{I}_4\cap {\cal K},\]
where ${\cal K}=\mathbb{Q}[a_0,\ldots,a_{n-1}]$. We proceed to show that
$\langle D_2\rangle=\ideal{I}_4\cap {\cal K}$.

Since $E=\res(f, \res(f_1, f_3, y), x)$, $E\in\ideal{I}_4\cap {\cal K}$. Let $(\bar{a}_0,\ldots,\bar{a}_{n-1},\bar{x},\bar{y},\bar{w})$ be any zero of
$\ideal{I}_4$ and $h$ be any polynomial in $\ideal{I}_4\cap {\cal K}$. Then
$$E(\bar{a}_0,\ldots,\bar{a}_{n-1})=h(\bar{a}_0,\ldots,\bar{a}_{n-1})=0.$$
By Theorem~\ref{thm:D2in},
$$D_2(\bar{a}_0,\ldots,\bar{a}_{n-1})=0,$$
so $D_2$ and $h$ have a nonconstant common divisor.
As $D_2$ is irreducible over $\mathbb{Q}$, $D_2\mid h$.

On the other hand, by Corollary \ref{cor:D2in}
\[D_2\in \left\langle f(x), f(y), f\left(\dfrac{x+y}{2}\right),w(x-y)-1\right\rangle = \ideal{I}_1=\ideal{I}_4.\]
Since $D_2\mid h$ for any $h\in\ideal{I}_4\cap{\cal K}$, the intersection $\ideal{I}_4\cap{\cal K}$ is a principal ideal generated by $D_2$. Therefore,
\[\langle D_2\rangle=\ideal{I}_4\cap {\cal K}=\ideal{I}_1\cap {\cal K}.\]

\vspace{-0.5cm}
\end{proof}

\begin{proposition}\label{propD2a}
Let $s_i$ be the elementary symmetric polynomial of degree $i$ in $r_1,\ldots,r_n$ and let $v_i=a_{n-i}-(-1)^{i}s_i$ for $i=1,\ldots,n$. Then
\begin{align*}
\Big\langle\prod_{
\scriptsize{i< j\neq k}}&(2\,r_k-r_i-r_j),~
v_1,\ldots,v_{n}\Big\rangle\cap \mathbb{Q}[a_0,\ldots,a_{n-1}]\\
=&\,\left\langle 2\,r_1-r_2-r_3,~
v_1,\ldots,v_{n}\right\rangle\cap \mathbb{Q}[a_0,\ldots,a_{n-1}]
=\left\langle D_2\right\rangle,
\end{align*}
where $D_2$ is the second discriminant of a monic univariate polynomial of degree $n$ with coefficients $a_0,\ldots,a_{n-1}$.
\end{proposition}

\begin{proof}
Let ${\cal K}=\mathbb{Q}[a_0,\ldots,a_{n-1}]$ as before and
\begin{align*}
\ideal{J}_1=&\left\langle\prod_{
\scriptsize{i< j\neq k}}(2\,r_k-r_i-r_j),~
v_1,\ldots,v_{n}\right\rangle\cap {\cal K},\\
\ideal{J}_2=&\,\left\langle 2\,r_1-r_2-r_3,~
v_1,\ldots,v_{n}\right\rangle\cap {\cal K}.
\end{align*}

Proof of $\ideal{J}_1=\left\langle D_2\right\rangle$.
Note first that each $v_i~(1\leq i\leq n)$ is a polynomial monic and linear in $a_{n-i}$. Dividing $D_2$ by $v_n,\ldots,v_1$ with respect to $a_0,\ldots,a_{n-1}$ respectively, one can obtain a remainder $R$ in $r_1,\ldots, r_n$. Then there exist polynomials $A_1,\ldots,A_n\in\mathbb{Q}[a_0,\ldots,a_{n-1},r_1,\ldots,r_n]$ such that
\[D_2=A_1v_1+\cdots+A_nv_n+R.\]
Substituting $a_{n-i}=(-1)^is_i$ into the above formula and by Theorem \ref{thm:D2}, we have
\[R=\prod_{\scriptsize{i< j\neq k}}(2\,r_k-r_i-r_j).\]
Therefore, $D_2$ can be written as a linear combination of polynomials in $\ideal{J}_1$. This implies that
$D_2\in \ideal{J}_1$
and thus $\left\langle D_2\right\rangle\subset\ideal{J}_1$.

To show that $\ideal{J}_1\subset\left\langle D_2\right\rangle$, let $h$ be any polynomial in $\ideal{J}_1$.
Then the greatest common divisor $\gcd(h,D_2)$ of $h$ and $D_2$ is contained in the ideal $\ideal{J}_1$. As $D_2$ is irreducible over ${\Bbb Q}$, $\gcd(h,D_2)$ is either a nonzero constant, or equal to $D_2$. If $\gcd(h,D_2)$ is a nonzero constant, then $\ideal{J}_1$ is equal to the unit ideal, which is not possible because for any $r_1,\ldots,r_n$ satisfying
$\prod_{
\scriptsize{i< j\neq k}}(2\,r_k-r_i-r_j)=0$, there always exist
$a_0,\ldots,a_{n-1}$ such that $v_1=\cdots=v_{n}=0$, i.e., $\ideal{J}_1$ always has zeros. Therefore, $\gcd(h,D_2)=D_2$ and $D_2\mid h$. It follows that $h\in\left\langle D_2\right\rangle$.

Proof of $\ideal{J}_1=\ideal{J}_2$.
Since $\ideal{J}_1\subset\ideal{J}_2$ holds obviously, we only need to show that $\ideal{J}_2\subset\ideal{J}_1$.
Observe that
\[
\ideal{J}_1\supset\bigcap_{i< j\neq k}\langle 2\,r_k-r_i-r_j,\,
v_1,\ldots,v_{n}\rangle\cap {\cal K}.
\]
Since $\ideal{J}_1=\langle D_2\rangle$ is a prime ideal, there exist $i< j\neq k$ such that
\[\ideal{J}_1\supset\ideal{J}_{ijk}=\langle 2\,r_k-r_i-r_j,~
v_1,\ldots,v_{n}\rangle\cap {\cal K}\]
and $\ideal{J}_{ijk}$ is prime.
Note that $v_1,\ldots,v_n$ are symmetric in $r_1,\ldots,r_n$. Hence the primality of $\ideal{J}_{ijk}$ implies the primality of $\ideal{J}_{\mu \nu \kappa}$ for all $\mu <\nu \neq\kappa$.  Therefore, all the $\ideal{J}_{\mu \nu \kappa}$ are identical. Hence
\[\ideal{J}_1\supset\ideal{J}_2=\langle 2\,r_1-r_2-r_3,~
v_1,\ldots,v_{n}\rangle\cap {\cal K}.\]

\vspace{-0.5cm}
\end{proof}

As shown by Propositions~\ref{propD2} and \ref{propD2a}, there are several
ideals with different generators whose intersections with ${\cal K}$ are equal to
$\langle D_2\rangle$. The generator $D_2$ of the principal ideal $\langle D_2\rangle$, which is an \emph{elimination ideal} of ${\cal I}_1=\cdots={\cal I}_4$ or ${\cal J}_1={\cal J}_2$, can be obtained by computing the reduced lexicographical Gr\"obner basis of any of the ideals ${\cal I}_{\mu }$ and ${\cal J}_{\nu }$ (see \cite[Lemma 6.8]{B1985G}).

\section{Degrees of Some Determinant Polynomials}\label{sec:determinant}

The two determinant polynomials $H$ and $F=\res(f_1, f_3, y)$, defined in Theorem \ref{thm:D2} and Proposition \ref{prop:resD2} respectively, can be used for the construction of the second discriminant $D_2$. In what follows, we provide some simple formulas for the exact degrees of $H$ and $F$ in $x$, which may be used for complexity analysis of $D_2$.

\begin{lemma}\label{lem:degree_bound_H}
Let $n$ be the degree of a univariate polynomial $f$ and $H$ be as in Theorem \ref{thm:D2}. Then $\deg(H,x)\leq (n-1)(n-2)/2$.
\end{lemma}

\begin{proof}
Let $g_1$, $g_3$ and $g_1^*$, $g_3^*$ be as in Lemma \ref{lem:resH} and its proof. Then \[G=\res(g_1,g_3,y)=[\res(g_1^*(x,z),g_3^*(x,z),z)]^2=H^2.\] Now consider
\[\Delta (\bm{a},x,y,\alpha)=\left|
\begin{array}{cc}\smallskip
g_1(x,y)&g_3(x,y)\\
g_1(x,\alpha)&g_3(x,\alpha)
\end{array}
\right|\bigg/(y-\alpha)\]
and let $\bm{\nu}=(x, y,\alpha)$ and $\tilde{n}=2\,\lfloor\frac{n-1}{2}\rfloor$.
It is easy to see that $\Delta$ is of degree
$2\,n-4$ in $\bm{\nu}$ and
$\deg(\Delta,\alpha)=\deg(\Delta,y)= \tilde{n}-1$, $\deg(\Delta,\bm{a})\leq 2$.
Let $\Delta$ be written as
\[\Delta=\sum_{\scriptsize\begin{array}{c}\scriptsize 0 \leq i\leq 2\,n-4\\ \scriptsize 0 \leq j,k\leq \tilde{n}-1
\end{array}}\delta_{ijk}x^iy^j\alpha^k.\]
Then for every term $x^iy^j\alpha^k$ occurring in $\Delta$, $i+j+k\leq \deg(\Delta, \bm{\nu})=2\,n-4$, so $i\leq 2\,n-4-j-k$.

Denote by \[B=(b_{j+1,k+1})=\left(\sum_{i=0}^{2\,n-4}\delta_{ijk}x^i\right)\] the $\tilde{n}\times \tilde{n}$ B\'ezout matrix of $g_1$ and $g_3$ with respect to $y$. It follows that
$$\deg(b_{jk},x)\leq 2\,n-2-j-k.$$
Let $(k_1,\ldots,k_{\tilde{n}})$ denote an arbitrary permutation of $(1,\ldots, \tilde{n})$ and $\bar{G}=\det(B)$.
Then
\[\begin{array}{rl}\smallskip
2\,\deg(H,x)\!\!\!&\displaystyle=\deg(H^2,x)=\deg(G,x)=\deg(\bar{G},x)\\[10pt]
&\displaystyle\leq\max_{(k_1,\ldots,k_{\tilde{n}})}\deg(b_{1k_1}\cdots b_{\tilde{n}~k_{\tilde{n}-1}},x) \\[10pt]
&=\max_{(k_1,\ldots,k_{\tilde{n}})}\sum_{j=1}^{\tilde{n}}\deg(b_{jk_j},x) \\[10pt]
&\displaystyle \leq \max_{(k_1,\ldots,k_{\tilde{n}})}\sum_{j=1}^{\tilde{n}}[2\,n-2-j-k_j] \\[10pt]
&=(2\,n-2)\cdot\tilde{n}-\sum_{j=1}^{\tilde{n}}j-\min_{(k_1,\ldots,k_{\tilde{n}})}
\sum_{j=1}^{\tilde{n}}k_j \\[10pt]
&\displaystyle =(2\,n-2)\cdot\tilde{n}-(\tilde{n}+1)\cdot\tilde{n}=(n-1)(n-2).
\end{array}
\]
Therefore, $\deg(H,x)\leq(n-1)(n-2)/2$.
\end{proof}

Lemma \ref{lem:degree_bound_H} provides an upper bound for $\deg(H, x)$. In what follows, we show that the bound can be achieved for a particular polynomial. Thus the degree of $H$ constructed from the generic form of $f$ is equal to the bound.

\begin{lemma}\label{lem:degree_specialH}
For $a_0=\cdots=a_{n-1}=0$, $\deg(H, x)={(n-1)(n-2)}/{2}$, where $n$ is the degree of a monic univariate polynomial $f$ with coefficient $a_0,\ldots,a_{n-1}$ and $H$ is as in Theorem \ref{thm:D2}.
\end{lemma}

\begin{proof}
When $a_0=\cdots=a_{n-1}=0$, ${H}$ becomes the $(n-2)$th leading principal minor of the following matrix
\[\left(
\begin{array}{ccccc}\smallskip
C_n^2x^{n-2}&C_n^4x^{n-4}&C_n^6x^{n-6}&\cdots\!&\!\cdots\\ [3pt]
C_n^1x^{n-1}&C_n^3x^{n-3}&C_n^5x^{n-5}&\cdots\!&\!\cdots\\ [3pt]
0&C_n^2x^{n-2}&C_n^4x^{n-4}&\cdots\!&\!\cdots\\ [3pt]
0&C_n^1x^{n-1}&C_n^3x^{n-3}&\cdots\!&\!\cdots\\[3pt]
\vdots&\vdots&\vdots&\ddots\!&\!\\[-10pt]
\vdots&\vdots&\vdots&\!&\!\hspace{-4pt}\raisebox{0.1cm}{\mbox{$\ddots$}}
\end{array}
\right).\]
Simple calculation shows that
\begin{align*}
&\!\!\!\!\!\!\!\!\!\!\!\!\!\!\!\!\!\!\!\!\!\!\!\!\!\!\!\!\!\!\!\!\!\!\!\!\!\!\!
\!\!\!\!\!\!\!\!\!\!\!\!\!\!\!\!\!\!\!
\left|
\begin{array}{ccccc}\smallskip
C_n^2x^{n-2}&C_n^4x^{n-4}&C_n^6x^{n-6}&\cdots\!&\!\cdots\\ [3pt]
C_n^1x^{n-1}&C_n^3x^{n-3}&C_n^5x^{n-5}&\cdots\!&\!\cdots\\ [3pt]
0&C_n^2x^{n-2}&C_n^4x^{n-4}&\cdots\!&\!\cdots\\ [3pt]
0&C_n^1x^{n-1}&C_n^3x^{n-3}&\cdots\!&\!\cdots\\ [3pt]
\vdots&\vdots&\vdots&\ddots\!&\!\\[-11pt]
\vdots&\vdots&\vdots&\!&\!\hspace{-6pt}\raisebox{0.1cm}{\mbox{$\ddots$}}
\end{array}
\right|
\end{align*}
\begin{align*}
\xlongequal[\substack{{\rm ro}_2\div x \\[2pt] \ldots \\[2pt] {\rm ro}_{n-2}\div x^{n-3}}]{\substack{{\rm co}_1\div x^{n-2}\\[2pt] {\rm co}_2\div x^{n-4}, {\rm co}_{n-2}\times x^{n-4}\\[2pt] {\rm co}_3\div x^{n-3}, {\rm co}_{n-3}\times x^{n-3}\\[2pt] \cdots}}
&\left|
\begin{array}{ccccc}\smallskip
C_n^2&C_n^4&C_n^6&\cdots\!&\!\cdots\\ [3pt]
C_n^1&C_n^3&C_n^5&\cdots\!&\!\cdots\\ [3pt]
0&C_n^2&C_n^4&\cdots\!&\!\cdots\\ [3pt]
0&C_n^1&C_n^3&\cdots\!&\!\cdots\\ [3pt]
\vdots&\vdots&\vdots&\ddots\!&\!\\[-11pt]
\vdots&\vdots&\vdots&\!&\!\hspace{-6pt}\raisebox{0.1cm}{\mbox{$\ddots$}}
\end{array}
\right|x^{(n-2)+1+2+\cdots+(n-3)}\\ \smallskip
=&\left|
\begin{array}{ccccc}\smallskip
C_n^2&C_n^4&C_n^6&\cdots\!&\!\cdots\\ [3pt]
C_n^1&C_n^3&C_n^5&\cdots\!&\!\cdots\\ [3pt]
0&C_n^2&C_n^4&\cdots\!&\!\cdots\\ [3pt]
0&C_n^1&C_n^3&\cdots\!&\!\cdots\\ [3pt]
\vdots&\vdots&\vdots&\ddots\!&\!\\[-11pt]
\vdots&\vdots&\vdots&\!&\!\hspace{-6pt}\raisebox{0.1cm}{\mbox{$\ddots$}}
\end{array}
\right|x^{\frac{(n-1)(n-2)}{2}}\doteq c_{n}x^{\frac{(n-1)(n-2)}{2}}.
\end{align*}
In what follows, we prove that $c_n\neq 0$.

Let
\begin{align*}
U&=C_n^2z^2+C_n^4z^4+\cdots+C_n^{2\left\lfloor\frac{n}{2}\right\rfloor}z^{2\left\lfloor\frac{n}{2}\right\rfloor},\\[8pt]
V&=C_n^1z+C_n^3z^3+\cdots+C_n^{2\left\lfloor\frac{n+1}{2}\right\rfloor -1}z^{2\left\lfloor\frac{n+1}{2}\right\rfloor-1},
\end{align*}
and $\bar{U}$ and $\bar{V}$ be obtained from $U/z^2$ and $V/z$, respectively, by replacing $z^2$ with $t$. Then $c_n=\pm\,\res(\bar{U}, \bar{V}, t)$.
If $c_n=0$, then $U/z^2$ and $V/z$ have at least one common zero, say $\bar{z}$, where $\bar{z}\neq0$. Note that
\[U+V=(z+1)^n-C_n^0.\]
Substituting $z=\bar{z}$ into the above equation, we have
$(\bar{z}+1)^n-1=0$.
Similarly, $$(U-V)|_{z=\bar{z}}=(\bar{z}-1)^n-1=0.$$
Therefore, there exist two unit roots $u_1, u_2$ such that $\bar{z}+1=u_1$ and $\bar{z}-1=u_2$, which leads to $u_1-u_2=2$. In other words, $u_1$ and $u_2$ have the same imaginary part and the difference of their real parts is $2$. This can happen only when $u_1=1$ and $u_2=-1$. Therefore, $\bar{z}=0$, which leads to contradiction since $\bar{z}$ is nonzero. Hence the conclusion holds.
\end{proof}

The following theorem follows from Lemmas \ref{lem:degree_bound_H} and \ref{lem:degree_specialH}.
\begin{theorem}
Let $f$ and $H$ be as in Theorem \ref{thm:D2} and ￥$n$ be the degree of $f$. Then
$$\deg(H, x)= (n-1)(n-2)/2.$$
\end{theorem}

Similarly, we have the following theorem.

\begin{theorem}\label{thm:deg_res_f1f3}
Let $f$ and $F$ be as in Theorem \ref{thm:qIs2} and ￥$n$ be the degree of $f$. Then $$\deg(F, x)= (n-1)(n-2).$$
\end{theorem}

This theorem is established by proving the following two lemmas.

\begin{lemma}\label{lem:deg_res_f1f3}
Let $f$ and $F$ be as in Theorem \ref{thm:qIs2} and ￥$n$ be the degree of $f$. Then $$\deg(F,x)\leq (n-1)(n-2).$$
\end{lemma}

\begin{proof}
Let
\[\Delta (\bm{a},x,y,\alpha)=\left|
\begin{array}{cc}\smallskip
f_1(x,y)&f_3(x,y)\\
f_1(x,\alpha)&f_3(x,\alpha)
\end{array}
\right|\bigg/(y-\alpha)\]
and $\bm{\nu}=(x, y,\alpha)$.
It is easy to see that $\Delta$ is of degree $2\,n-4$ in $\bm{\nu}$ and
$$\deg(\Delta,\alpha)=\deg(\Delta,y)= n-2,$$ $\deg(\Delta,\bm{a})\leq 2$.
Let $\Delta$ be written as
\[\Delta=\sum_{\scriptsize\begin{array}{c}\scriptsize 0 \leq i\leq 2n-4\\ \scriptsize 0 \leq j,k\leq n-2
\end{array}}\delta_{ijk}x^iy^j\alpha^k.\]
Then for every term $x^iy^j\alpha^k$ occurring in $\Delta$,
$$i+j+k\leq \deg(\Delta, \bm{\nu})=2\,n-4,$$
so $i\leq 2\,n-4-j-k$.

Denote by \[B=(b_{j+1,k+1})=\left(\sum_{i=0}^{2\,n-4}\delta_{ijk}x^i\right)\] the $(n-1)\times (n-1)$ B\'ezout matrix of $f_1$ and $f_3$ with respect to $y$. It follows that $$\deg(b_{jk},x)\leq 2\,n-2-j-k.$$
Let $(k_1,\ldots,k_{n-1})$ denote an arbitrary permutation of $(1,\ldots, n-1)$ and $\bar{F}=\det(B)$. According to the theory of resultants \cite{B1779T}, $F=\res(f_1,f_3,y)=\pm\, \bar{F}$. Therefore,
\[\begin{array}{rl}\medskip
\deg(F,x)\!\!\!&\displaystyle=\deg(\bar{F},x)\leq\max_{(k_1,\ldots,k_{n-1})}\deg(b_{1k_1}\cdots b_{n-1,k_{n-1}},x) =\max_{(k_1,\ldots,k_{n-1})}\sum_{j=1}^{n-1}\deg(b_{jk_j},x) \\ \medskip
&\displaystyle \leq \max_{(k_1,\ldots,k_{n-1})}\sum_{j=1}^{n-1}(2\,n-2-j-k_j)
=(2\,n-2)(n-1)-\sum_{j=1}^{n-1}j-\min_{(k_1,\ldots,k_{n-1})}
\sum_{j=1}^{n-1}k_j \\
&\displaystyle =2\,(n-1)^2-(n-1)n=(n-1)(n-2).
\end{array}
\]

\vspace{-0.5cm}
\end{proof}

\begin{lemma}
For $a_0=\cdots=a_{n-1}=0$, $\deg(F,x)=(n-1)(n-2)$, where $n$ is the degree of a monic univariate polynomial $f$ with coefficient $a_0,\ldots,a_{n-1}$ and $F$ is as in Theorem \ref{thm:qIs2}.
\end{lemma}

\begin{proof}
When $a_0=\cdots=a_{n-1}=0$, $f=x^n$.
We first prove that $x=0$ is equivalent to $F=0$.

($\Longrightarrow$) If $x=0$, then $f_1=y^{n-1}$ and $f_3=\left(2-{1}/{2^{n-2}}\right)y^{n-2}$. In this case, $f_1$ and $f_3$ have a common zero and thus $F=0$.

($\Longleftarrow$) Let $F=0$; then $f_1$ and $f_3$ have at least one common zero for $y$, say $\bar{y}$. Then
\[
f_1(x,\bar{y})=\dfrac{\bar{y}^n-x^n}{\bar{y}-x}=0, \quad f_3(x,\bar{y})=\dfrac{\bar{y}^n-2\,\left(\dfrac{\bar{y}+x}{2}\right)^n+x^n}{\dfrac{(\bar{y}-x)^2}{2}}=0.
\]
Suppose that $x\neq 0$ and let $\bar{t}=\bar{y}/x$. Then the above equalities imply that
\[\bar{t}^n=1,\quad\left(\frac{1}{2}+\frac{1}{2}\bar{t}\right)^n=1.\]
Therefore, there exist two unit roots $u_1$ and $u_2$ such that $\bar{t}=u_1$ and $(1+\bar{t})/2=u_2$, which implies that $u_2=(1+u_1)/2$. This can happen only when $u_1=u_2=1$; so $\bar{y}=x$. Thus
\[f_1(x,\bar{y})=\dfrac{y^n-x^n}{y-x}=x^{n-1}+x^{n-2}\bar{y}+\cdots+x\bar{y}^{n-2}+x\bar{y}^{n-1}=nx^{n-1}=0,\]
which implies that $x=0$. This contradicts the assumption that $x\neq0$. Therefore, $x=0$.

Since $x=0$ and $F=0$ are equivalent, there exist a nonzero constant $c$ and an integer $N\geq1$ such that $F=c\,x^N$. It remains to show that $N=(n-1)(n-2)$.

Let
\[\Delta (x,y,\alpha)=\left|
\begin{array}{cc}\smallskip
f_1(x,y)&f_3(x,y)\\
f_1(x,\alpha)&f_3(x,\alpha)
\end{array}
\right|\bigg/(y-\alpha)\]
and $\bm{\nu}=(x, y,\alpha)$.
It is easy to see that $\Delta$ is homogeneous of degree $2\,n-4$ in $\bm{\nu}$ and
$\deg(\Delta,\alpha)=\deg(\Delta,y)= n-2$.
Let $\Delta$ be written as
\[\Delta=\sum_{\scriptsize\begin{array}{c}\scriptsize 0 \leq i\leq 2n-4\\ \scriptsize 0 \leq j,k\leq n-2
\end{array}}\delta_{ijk}x^iy^j\alpha^k.\]
Then for every term $x^iy^j\alpha^k$ occurring in $\Delta$, $i+j+k= \deg(\Delta, \bm{\nu})=2\,n-4$, so $i=2\,n-4-j-k$.

Denote by \[B=(b_{j+1,k+1})=\left(\sum_{i=0}^{2\,n-4}\delta_{ijk}x^i\right)\] the $(n-1)\times (n-1)$ B\'ezout matrix of $f_1$ and $f_3$ with respect to $y$ and let $\bar{F}=\det(B)$. According to the theory of resultants \cite{B1779T}, $F=\res(f_1,f_3,y)=\pm\, \bar{F}$, so $\deg(\bar{F},x)\geq1$.

Note that for any entry $b_{jk}$ in $B$, either $b_{jk}=0$ or $\deg(b_{jk},x)= 2\,n-2-j-k$. Let $(k_1,\ldots,k_{n-1})$ denote an arbitrary permutation of $(1,\ldots, n-1)$.
Then either $b_{1k_1}\cdots b_{n-1,k_{n-1}}=0$, or
\begin{align*}
\deg(b_{1k_1}\cdots b_{n-1,k_{n-1}},x)&=\sum_{j=1}^{n-1}\deg(b_{jk_j},x)=\sum_{j=1}^{n-1}(2\,n-2-j-k_j)\\
&=(2\,n-2)(n-1)-\sum_{j=1}^{n-1}j-\sum_{j=1}^{n-1}k_j \\
&=2\,(n-1)^2-(n-1)n=(n-1)(n-2).
\end{align*}

Note that $\bar{F}\neq 0$, so $\deg(\bar{F},x)=(n-1)(n-2)$. It follows that $\deg(F,x)=(n-1)(n-2)$.
\end{proof}

The following lemma has been used for the proof of Theorem \ref{thm:qIs2}.
\begin{lemma}\label{lem:deg_res_ff1f3}
Let $n$ be the degree of a monic univariate polynomial with coefficients
$\bm{a}=(a_0,\ldots,a_{n-1})$ and $E$ be as in Theorem \ref{thm:qIs2}.
Then $\deg(E,\bm{a})\leq 3\,(n-1)(n-2)+2\,(n-2)$.
\end{lemma}

\begin{proof}
Let $N=\deg(F,x)$; then $N\leq(n-1)(n-2)$ according to Lemma \ref{lem:deg_res_f1f3}. Moreover, from the proof of Lemma \ref{lem:deg_res_f1f3} we know that $\deg(F,\bm{a})\leq 2\,(n-2)$. Since $E$ is a determinant formed with $N$ rows of $f$-coefficients and $n$ rows of $F$-coefficients, the degree of each $f$-coefficient is at most $1$, and the degree of each $F$-coefficient is at most $2\,(n-2)$, the degree of $E$ is at most $N\cdot 1+n\cdot 2\,(n-2)\leq 3\,(n-1)(n-2)+2\,(n-2)$. The proof is complete.
\end{proof}

From Proposition \ref{prop:degD2} and Theorem \ref{thm:qIs2} the following corollary follows.

\begin{corollary}
Let $n$ be the degree of a monic univariate polynomial with coefficients
$\bm{a}=(a_0,\ldots,a_{n-1})$ and $E$ be as in Theorem \ref{thm:qIs2}. Then $\deg(E,\bm{a})= 3\,(n-1)(n-2)$.
\end{corollary}

The result of this corollary allows us to reduce the upper bound
$$3\,(n-1)(n-2)+2\,(n-2)$$
of $\deg(E,\bm{a})$ to $3\,(n-1)(n-2)$, the exact degree of $E$ in $\bm{a}$, which is also the degree of $D_2^2$ in $\bm{a}$.

\begin{remark}
{The determinant polynomials $F$ and $H$ are both irreducible over $\mathbb{Q}[\bm{a}]$. The irreducibility of $H$ is obvious because $D_2=\res(f, H,x)$ is irreducible and that of $F$ can be proved by using the symmetry of $F(r_1)$ with respect to $r_2,\ldots,r_n$.\footnote{Let $F(\bm{a},x)=F_1(\bm{a},x)F_2(\bm{a},x)$ with $\deg(F_1,x)\neq 0$. In this equality, substitution of $x$ by $r_1$ and elimination of each $a_i$ by using Vieta's formula yield $\bar{F}(r_1,\ldots,r_n)=\bar{F}_1(r_1,\ldots,r_n)\bar{F}_2(r_1,\ldots,r_n)$, where $\bar{F}$, $\bar{F}_1$, and $\bar{F}_2$ are all symmetric with respect to $r_2,\ldots,r_n$.
From the proof of Theorem \ref{thm:qIs2}, one sees that $\bar{F}=c\prod_{k\neq j}(r_1-2\,r_k+r_j)$ for some constant $c$. Thus $\bar{F}_1$ has at least one divisor $r_1-2\,r_k+r_j$ for some $j\neq k$. The symmetry of $\bar{F}_1$ with respect to $r_2,\ldots,r_n$ implies that $\prod_{\scriptsize{1<k\neq j}}(r_1-2\,r_k+r_j)$ is also a divisor of $\bar{F}_1$. Therefore, $\bar{F}_1$ differs from $\bar{F}$ only by a nonzero constant, and so does $F_1$ from $F$. It follows that $F_2$ is a constant. This proves the irreducibility of $F$.}
Hence $F$ and $H$ do not have any common divisor. On the other hand, $G$ is obtained from $f_1$ and $f_3$ via linear transformation and resultant computation and $F$ is connected to $H$ via $G$ by the relations
\begin{align*}
\res(f,F,x)&=\res(f,G,x)=[\res(f, H, x)]^2,\\
\deg(F, x)&=\deg(G,x)=2\,\deg(H,x),
\end{align*}
and $G=H^2$. However, it is unclear whether there is any direct connection between $F$ and $G$. Note that $F$ and thus $D_2$ are constructed from $f$, $f_1$, and $f_3$ naturally; yet the occurrence of the sequences of odd derivatives and even derivatives of $f$ with respect to $x$ in the determinant expressions of $H$ and $G$ remains uninterpretable. Meaningful interpretations of the occurrence might be figured out by exploring direct connections between $F$ and $G$.}
\end{remark}

%====================================================================================
\section{Application and Remarks}
\label{sec:ApplicationRemarks}
%====================================================================================

In this section, we illustrate the usefulness of the second discriminant by an application (to the classification of root configurations for the cubic polynomial) and discuss the possibility of introducing discriminants of higher order.

The form $r_i-r_j$ in $D_1$ can be viewed as the vector from $r_j$ to $r_i$, considered as two points in the complex plane. Similarly, the form $2\,r_k-r_i-r_j$ in $D_2$ can be viewed as twice the vector from the middle point of $r_i$ and $r_j$ to $r_k$. The signs of $D_1$ and $D_2$ carry information about the distribution, position, and relative configuration of the roots $r_1,\ldots,r_n$ of $f$. Therefore, $D_1$ and $D_2$ can be used to explore such structural properties of the roots of $f$ without exactly computing them out.

For the cubic polynomial $f=x^3+a_2x^2+a_1x+a_0$, we have the following Lagrange formula with radicals for its three roots:
\[
r_1\,=\dfrac{-a_2+\omega^1c_1+\omega^2c_2}{3},\quad
r_2\,=\dfrac{-a_2+\omega^0c_1+\omega^2c_2}{3},\quad
r_3\,=\dfrac{-a_2+\omega^2c_1+\omega^1c_2}{3},\quad
\]
where $\omega=e^{\frac{2\pi}{3}{\rm i}}=-\frac{1}{2}+\frac{\sqrt{3}}{2}{\rm i}$ and
\[c_1=\sqrt[3]{(D_2+2\,\sqrt{-3\,D_1})/2},\quad c_2=\sqrt[3]{(D_2-2\,\sqrt{-3\,D_1})/2}.\]
Using the above formula, one can classify the roots of $f$ into 9 types of configurations according to the signs of $D_1$ and $D_2$ as shown in \cref{table:cubicconfig} (cf.\ \cite{ZWH2011S}).

\begin{table}[htbp]
\caption{Types of configurations for the roots $r_1,r_2,r_3$ of the cubic polynomial $f$, where Re($r_1$)\,$\geq\,$Re($r_2$)$\,\geq\,$Re($r_3$) and the red points of small, middle, and large sizes stand respectively for single, double, and triple roots of $f$.}\label{table:cubicconfig}
\begin{center}
\begin{tabular*}{0.92835\textwidth}{|c|c|c|c|}
\hline
&&&\\[-0.2cm]
&  $D_2<0$ & $D_2=0$      & $D_2>0$\\[0.1cm]
\hline
&&&\\
$D_1>0$
&\includegraphics[width=0.25\textwidth]{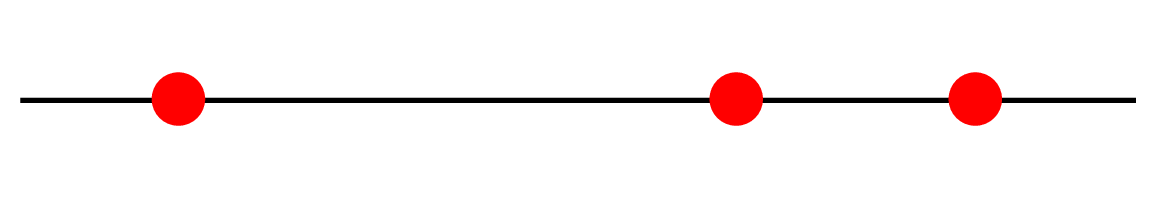}
&\includegraphics[width=0.25\textwidth]{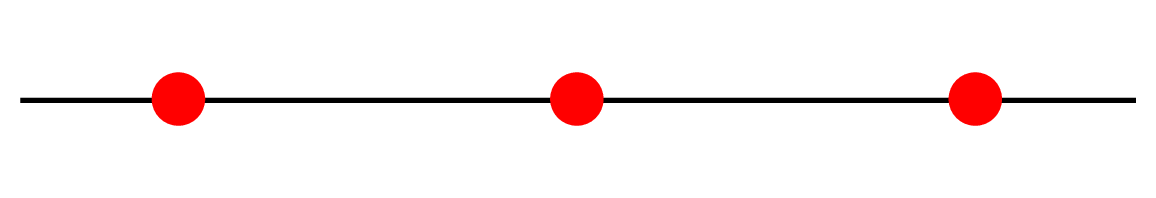}
&\includegraphics[width=0.25\textwidth]{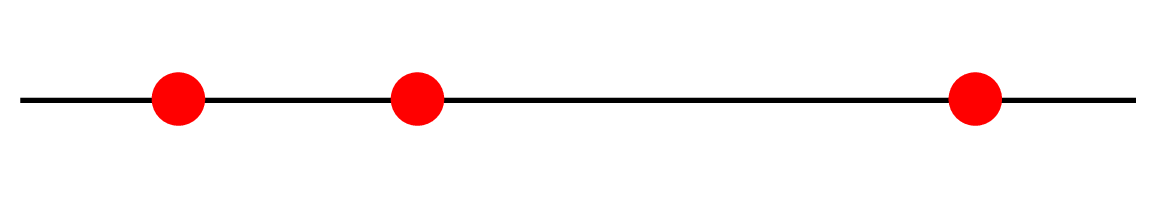}\\[0.2cm]
&$r_1,r_2,r_3\in\mathbb{R}$,&$r_1,r_2,r_3\in\mathbb{R}$, &$r_1,r_2,r_3\in\mathbb{R}$,\\
&$r_1-r_2<r_2-r_3$ & $r_1-r_2=r_2-r_3$ &$r_1-r_2>r_2-r_3$\\
\hline
&&&\\
$D_1=0$
&\raisebox{-0.4\height}{\includegraphics[width=0.25\textwidth]{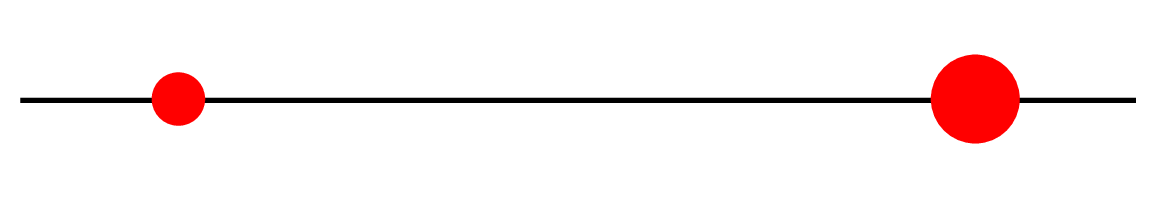}}
&\raisebox{-0.4\height}{\includegraphics[width=0.25\textwidth]{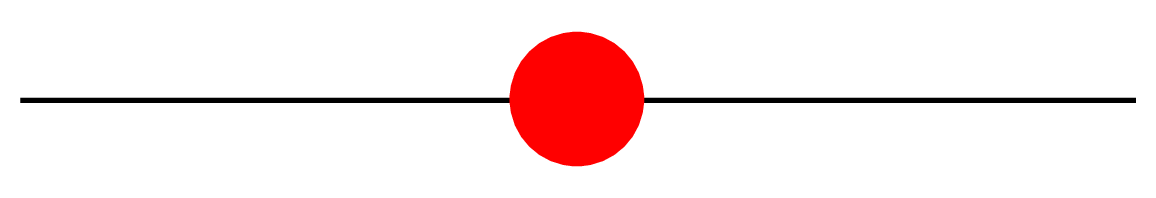}}
&\raisebox{-0.4\height}{\includegraphics[width=0.25\textwidth]{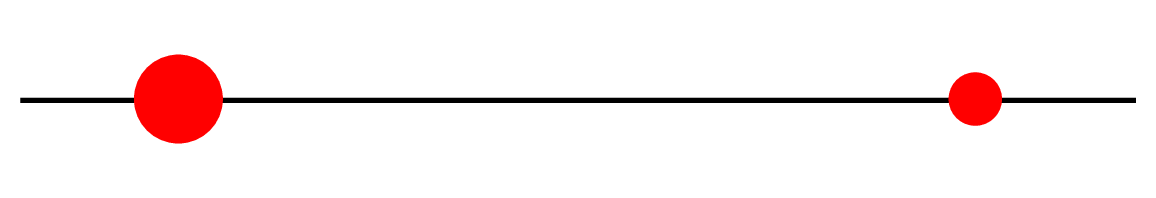}}\\[0.4cm]
&$r_1,r_2,r_3\in\mathbb{R}$ & $r_1,r_2,r_3\in\mathbb{R}$ & $r_1,r_2,r_3\in\mathbb{R}$  \\
& $r_1=r_2>r_3$ & $r_1=r_2=r_3$ &$r_1>r_2=r_3$\\
\hline
&&&\\[-0.2cm]
$D_1<0$
&\raisebox{-0.35\height}{\includegraphics[width=0.25\textwidth]{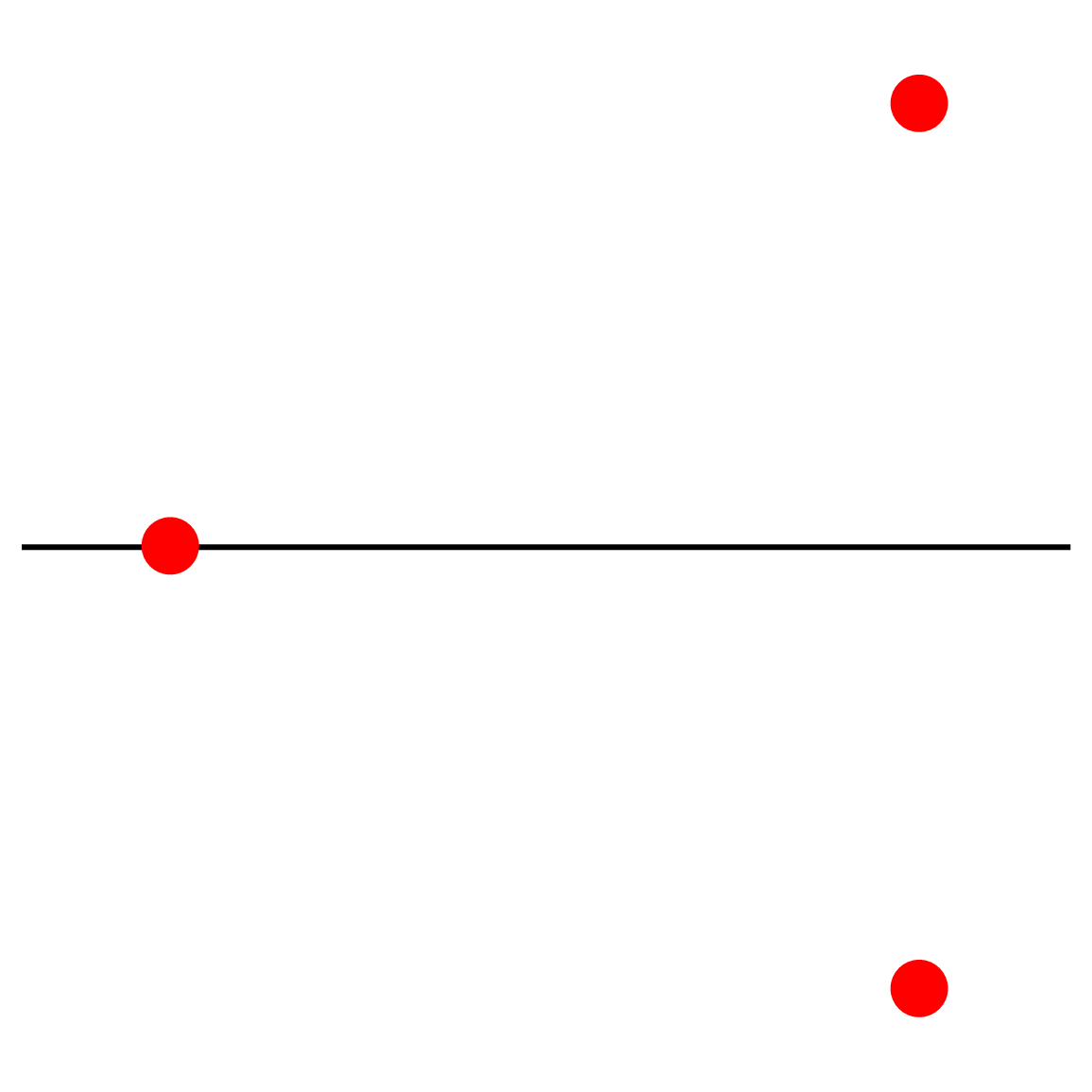}}
&\raisebox{-0.35\height}{\includegraphics[width=0.25\textwidth]{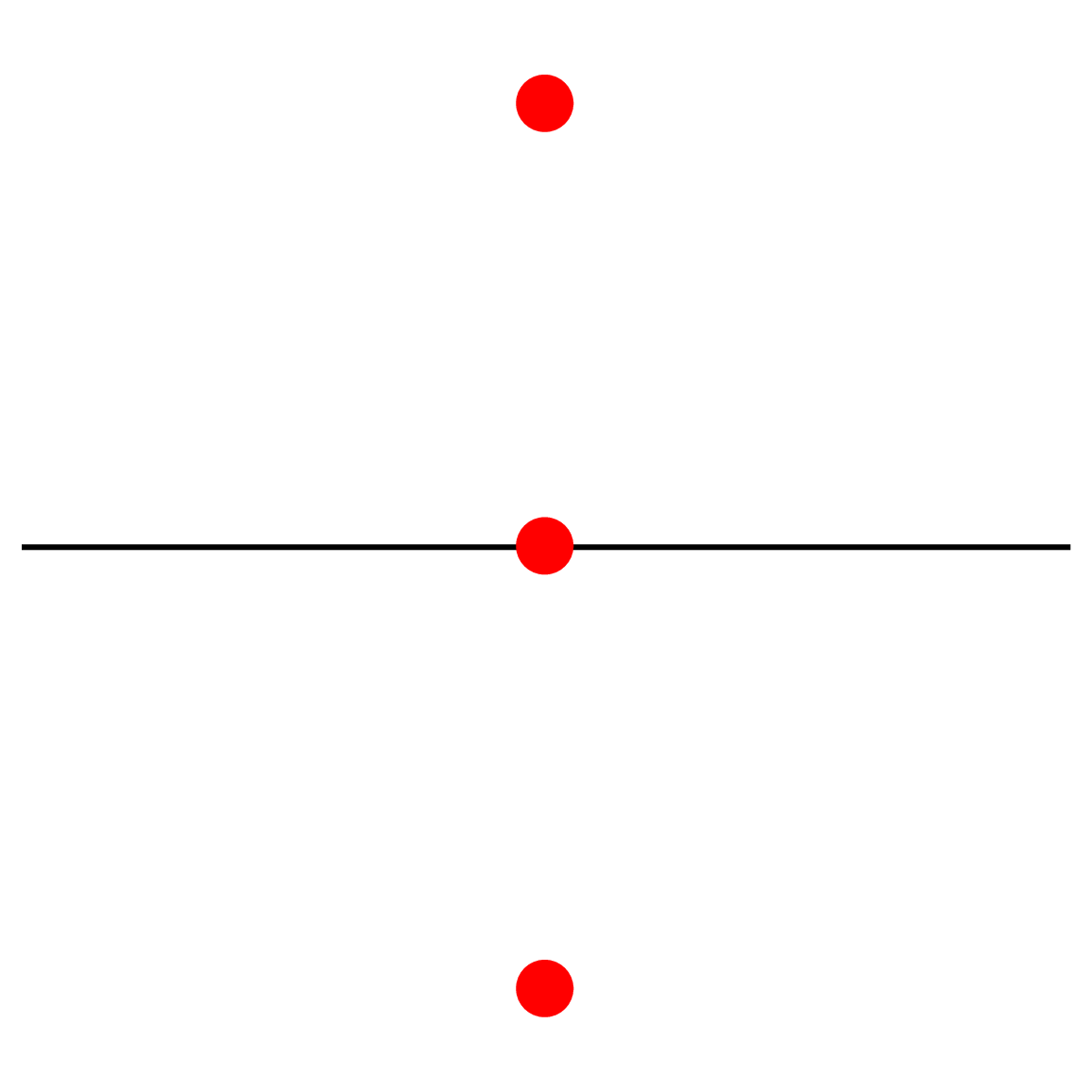}}
&\raisebox{-0.35\height}{\includegraphics[width=0.25\textwidth]{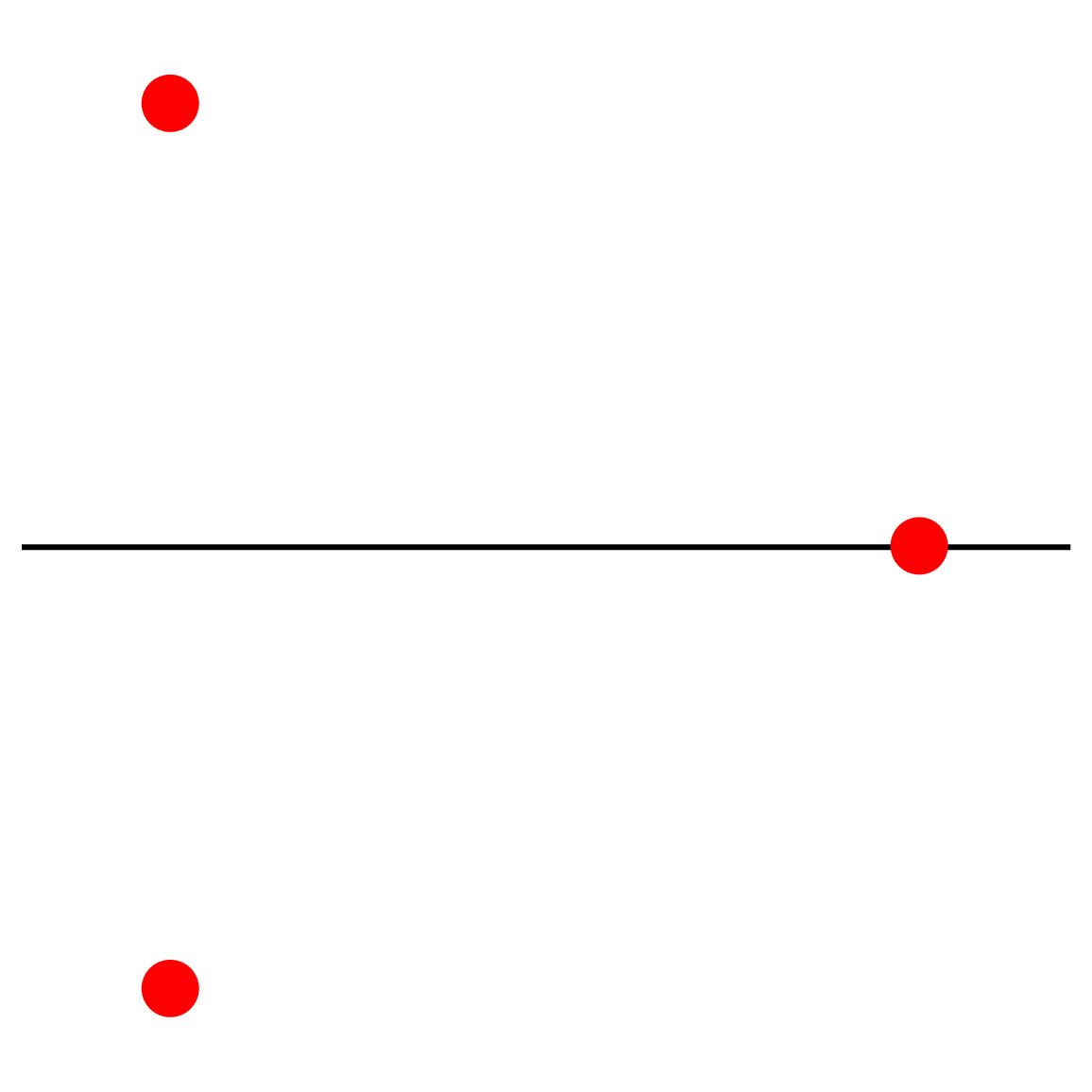}}\\
&&&\\[-0.2cm]
& $r_1,r_2\in\mathbb{C},r_3\in\mathbb{R}$ & $r_1,r_3\in\mathbb{C},r_2\in\mathbb{R}$ & $r_2,r_3\in\mathbb{C},r_1\in\mathbb{R}$ \\
& ${\rm Re}(r_1)={\rm Re}(r_2)>r_3$ &${\rm Re}(r_1)={\rm Re}(r_2)=r_3$ &$r_1>{\rm Re}(r_2)={\rm Re}(r_3)$  \\
\hline
\end{tabular*}
\end{center}
\end{table}

The second discriminant can also be used in the root formula with
radicals and to classify the types of configurations of the four
roots for the general quartic polynomial. The classification in this case
is somewhat involved and will be presented in a forthcoming paper \cite{HWYZ2016S}.

The second discriminant of a univariate polynomial $f$, a concept we have introduced, is defined as the product of all possible linear forms $2\,r_k-r_i-r_j$ in the roots $r_i,r_k,r_j$ of $f$ with $i<j\neq k$, so its vanishing is a necessary and sufficient condition for $f$ to have a symmetric triple of roots, i.e., a triple $(r_i, r_k, r_j)$ of roots of $f$ such that $r_k=(r_i+r_j)/2$. We have shown that the second discriminant of $f$ can be expressed as the resultant of $f$ and a determinant formed with the derivatives of $f$ and it possesses several notable properties\footnote{Our experiments also show that, when $a_0,\ldots,a_{n-1}$ take integer values,  $D_2\not\equiv 2 \mod 4$ for $n>3$.} and can be used to analyze the structure of the roots of $f$.

We may naturally consider the product of linear forms in $d$ roots of $f$ for any $n\geq d\geq 4$. The product should be symmetric with respect to the $n$ roots of $f$ and the linear form should be chosen such that its vanishing constrains the $d$ general roots of $f$
to form a degenerate configuration which is geometrically interesting.
Then one can try to establish conditions for $f$ to have $d$ roots forming the degenerate configuration.

For $n\geq d=4$, linear forms of interest in four roots $r_i, r_j, r_k, r_l$ of $f$ could be taken of the following type
\begin{equation}\label{d4av}
r_i+r_j-r_k-r_l,\quad \mbox{or} \quad 3\,r_l-r_i-r_j-r_k.\end{equation}
The former is twice the difference between the average of the two roots $r_i$ and $r_j$ and that of the two roots $r_k$ and $r_l$, while the latter
is three times the difference from the root $r_l$ to the average of the three
roots $r_i, r_j, r_k$. When the roots are considered as points in the complex
plane, the average of two or three roots may be interpreted as the middle point or the centroid of the two or three points, respectively.
Using the first linear form in \eqref{d4av}, one may define
\[D_3=\prod_{\scriptsize{\begin{array}{c}
{i\neq j\neq k\neq l}\\
i<j, k<l, i<k
\end{array}}}{(r_i+r_j-r_k-r_l)}.\]
For $n=4$, $D_3$ can be expressed as a polynomial in the coefficients of $f$ and this polynomial has been used in the root formula of $f$ with radicals. How to express $D_3$ as a polynomial in the coefficients of $f$ for arbitrary $n> 4$ and what properties $D_3$ may have are questions that remain for further investigation. Similar questions may be asked for $D_3$ defined by using the other linear form, and for $D_4$, $D_5$, \ldots, when they are properly defined.

It should be pointed out that the ideas and methodologies used in the study of $D_2$ provide a new approach to explore the properties of $D_1$. It may be generalized to investigate $D_3, D_4, \ldots$ and to discover other mysteries about the roots of $f$.

\Acknowledgements{Teo Mora pointed the master reference of 1906--1923 ``The History of Determinants in the Historical Order of Development"  by Sir Thomas Muir to us. We searched this reference (four volumes with more than 2\,000 pages) and other classic references in the literature and we could not find any notion nor any result related to our work.
Hoon  Hong pointed out that the proof of Theorem \ref{thm:D2} can be shortened  by using Orlando formula. We thank both Teo Mora  and Hoon  Hong for their personal communications and helpful comments. This work was supported by National Natural Science Foundation of China (Grant Nos. 61702025 and 11801101), the Special Fund for Guangxi Bagui Scholar Project,
Guangxi Science and Technology Program (Grant No. 2017AD23056),
and the Startup Foundation for Advanced Talents in Guangxi University
for Nationalities (Grant No 2015MDQD018).}


\end{document}